\documentclass[smallcondensed]{svjour3}
\usepackage{amsmath}

\usepackage{mathptmx}
\usepackage{amsfonts}
\usepackage{amssymb}
\usepackage{comment}
\usepackage{graphicx}
\usepackage{mathptmx}
\usepackage{marvosym}
\usepackage{url}
\usepackage{xcolor}
\usepackage[ruled,vlined,linesnumbered]{algorithm2e}
\usepackage{pstricks-add,pst-slpe}
\usepackage{pstricks}
\usepackage{pst-bar}
\usepackage{pstricks-add}
\usepackage{caption}
\usepackage{mathtools}
\usepackage{enumerate}
\usepackage{xparse}
\usepackage{calc}
\usepackage{tikz}
\usetikzlibrary{calc}

\usepackage{framed}

\usepackage{mdframed}
\usepackage{subcaption}
\usepackage{mwe}
\usepackage{etoolbox}
\let\bbordermatrix\bordermatrix
\patchcmd{\bbordermatrix}{8.75}{4.75}{}{}
\patchcmd{\bbordermatrix}{\left(}{\left[}{}{}
\patchcmd{\bbordermatrix}{\right)}{\right]}{}{}

\usepackage{empheq}
\usepackage[most]{tcolorbox}

\newtcbox{\mymath}[1][]{%
	nobeforeafter, math upper, tcbox raise base,
	enhanced, colframe=blue!30!black,
	colback=blue!30, boxrule=1pt,
	#1}

\SetCommentSty{mycommfont}

\newcommand{\vv}{V}
\newcommand{\ee}{E}

\newcommand{\yy}{\mathbf{y}}
\newcommand{\xx}{\mathbf{x}}

\newcommand{\QQ}{Q}
\newcommand{\CC}{\mathbb{C}}
\newcommand{\RR}{\mathbb{R}}
\newcommand{\PP}{P}
\newcommand{\NN}{N}

\def \HB {B}
\def \KK {K}
\def \SS {S}
\def \BB {{\mathbf{B}}}
\def \CC {{\mathbf{C}}}
\def \DD {{\mathbf{D}}}
\def \HH {{\mathbf{H}}}
\def \FF {{\mathbf{F}}}
\def \LL {l}
\def \BS {{\mathbf{S}}}
\def \KLU {{\tt KLU}}
\def \LAP {{\tt LAPACK}}

\def \incCG {{\tt incCG}}
\def \lapCG {{\tt lapackCG}}
\def \kluCG {{\tt kluCG}}

\def \II {{\mathbf{I}}}
\def \zero  {{\mathbf{0}}}
\def \XI {I}
\def \XJ {J}

\def \rhs {{\mathbf{b}}}

\def \aa {{\mathbf{\alpha}}}
\def \bb {{\mathbf{\beta}}}
\newcommand{\cc}{\mathbf{c}}
\def \uu {{\mathbf{\mu}}}

\def \ll {{\mathbf{\lambda}}}
\def \ww {w}
\def \PW {\omega^{\ww}}
\def \PPW {\omega^{\ww'}}
\def \dd {q} 

\def \ksi {{\mathbf{\xi}}}

\def \proj {\delta}

\def \UU {{\mathbf{M}}}

\def  \MCF {{{\bf MCF}}}

\def \MMCF  {{{\bf MMCF}}}
\def \MIN {{{\bf \min}}}
\def  \AA {{\mathbf{A}}}
\def \AIJ {{\AA_{\XI,\XJ}}}

\def \HC {C}

\def \SCG {{\tt SCG}}
\def \SMCG {{\tt SMCG}}
\def \locS {{\tt locSolver}}
\def \incS {{\tt incSolver}}

\def \reach {{\tt reach}}
\def \lift {{\tt lift}}

\newtheorem{mydef}{Definition}

\newtheorem{model}{Model}

\DeclareMathOperator*{\argmin}{arg\,min}

\graphicspath{{img/}}

\title{Solving Splitted Multi-Commodity Flow Problem by Efficient　 Linear Programming Algorithm }

\author{Liyun Dai \and  Hengjun Zhao \and Zhiming Liu }
\institute{Liyun Dai(\Letter)
	\at  RISE, Southwest University, Chongqing, China
	\\\email{dailiyun@swu.edu.cn}
	\and
	Hengjun Zhao
	\at   RISE, Southwest University, Chongqing, China
	\\\email{zhaohj2016@swu.edu.cn}
	\and
	Zhiming Liu
	\at   RISE, Southwest University, Chongqing, China
	\\\email{zhimingliu88@swu.edu.cn}
}

\begin{document} 
\label{firstpage}
\date{}
\maketitle

\begin{abstract}
Column generation is often used to solve  {\em multi-commodity flow problems}.  A program for  column generation  always includes a module that solves a linear equation.
In this paper, we address  three major issues in solving linear problem during column generation procedure which are  
  (1) how to employ the sparse property of the coefficient matrix;  (2) how to reduce the size of the  coefficient matrix;  and (3) how to reuse the solution to a similar equation.  To this end, we first analyze the sparse property 
   of  coefficient matrix of linear equations and find that the matrices occurring in iteration are very sparse.  Then, we present an algorithm \locS\  (for \emph{localized system solver})  for linear equations with sparse coefficient matrices and right-hand-sides. This algorithm can  reduce the number of variables.   After that, we present  the algorithm \incS\  (for \emph{incremental system solver})  which utilizes similarity  in the iterations of the program for a linear equation system. All three techniques can be used in  column generation
  of multi-commodity problems. Preliminary numerical experiments show that the \incS\   is significantly  faster than the existing algorithms.  For example, random test cases show that  \incS\  is at least 37 times and  up to   341 times faster than popular solver \LAP. 

\end{abstract}

\section{Introduction}
\label{sec:intro}
The  multi-commodity flow problem (\MCF) is a network flow problem with multiple commodities (flow demands) between different source and target nodes. Solving this problem is to  find an assignment to all the flow variables such that  certain given constraints are satisfied~\cite{ford2004a}.  Many application problems can be reduced to \MCF. 
 Examples of these applications include the vehicle routing problem 　({\bf VRP}) \cite{letchford2015stronger,cattaruzza2014an,moshrefjavadi2016the}, 
the traveling salesman problem ({\bf TSP}) \cite{hernandezperez2009the}, and 
problems of routing and wavelength assignment ({\bf RWA}) \cite{leesutthipornchai2010solving,patel2012routing}.
While it is well known that   offline network resource optimization and planing  in  traditional network is  a  typical \MCF\ \cite{Ahuja:1993,ford2004a,Jajszczyk2005},
online network resource optimization and planing, which are now widely regarded  as software
	defined network  ({\bf SDN}), are also treated as \MCF\ \cite{hong2013achieving,guo2014traffic,kandula2015calendaring}.

Because of its importance, there have been a sizable body of work on  \MCF, e.g. \cite{mahey2001capacity,DBLP:GargK07,holmberg2003a,Huisman2005,ZHU2012164,barnhart1998branch-and-price:,degraeve2007a,holmberg2003a,huisman2007a,salmasi2010total,Briant2008}, in which  {\em colummn generation} is widely used. A survey on column generation is given in \cite{lubbecke2005selected}. There, the algorithms  are  divided in two classes, which are  called {\em exact algorithms}              \cite{dantzig1967generalized,mccallum1977a,barnhart1994a,mamer2000a,holmberg2003a,dinh2013combining} and  {\em approximation algorithms} \cite{goldberg1992a,grigoriadis1994fast,awerbuch1994improved,Grigoriadis96,even1999fast,fleischer1999approximating,fleischer2002fast,bienstock2006approximating,karakostas2008faster}, respectively. In this paper, we will focus on the　 exact algorithm for the splitted multi-commodity flow problem in which  the flow demands  can be splitted among multiple paths for one commodity.

{\bf Organization:} After this introduction, we introduce \MCF\ and three different models for it in Section \ref{sec:pro}. In Section \ref{sec:review}, we give a summary on column generation for  \MCF. We show in Section \ref{sec:struct} how we  apply the result in \cite{dantzig1967generalized} to \MCF, and present a concrete block structure of the  basic matrix of column generation. In Section \ref{section:sparse}, we present the properties of the  coefficient matrix. The test  results show that  the number of nonzero elements in each row of the coefficient matrix is less than 5 even when the  length of the row is greater than $1000$. Thus, the matrix is very sparse.  We devote Section \ref{sec:advance} to present  the two algorithms that are our main contribution in this paper. The first algorithm, called \locS,  is a localized system solver. This algorithm can reduce the number of variables in solving a linear equation when both its coefficient matrix and right-hand-sides are sparse. The second algorithm, called \incS, is an incremental system solver which utilizes similarity during the iterations in solving linear equations. We present our experiment test results in Section \ref{sec:exp}, and conclusions in Section~\ref{section:conc}.
%
%
%


\section{Model for Multi-Commodity Flow Problem}
\label{sec:pro}

In this section, we define the basic formulation of multi-commodity flow  problem ({\bf Model 1} below). We then present two more models,  which are called  \emph{Node-Link Formulation} and \emph{Link-Path Formulation} of \MCF\, respectively.  Both are  
linear programming models  with a large numbers of  variables and   constraints.

\subsection{The Basic Model of  {\bf MCF}}
\label{sec:prob}

Graphs are the most fundamental mathematical models for networks, and their edges and/or nodes are associated with numerical functions for quantity based network control and management. The basic graph model used to represent a {\bf MCF}  is a direct graph with {\em capacities} and {\em weights} assigned to its edges, which are used to represent factors and elements of ``effectiveness" and   ``cost elements" of network resources, respectively. 

A \emph{capacitated and weighted network}  is a triple $\mathcal{N}=\left(G (\vv,\ee), \dd, \ww\right)$, where
\begin{itemize}
	\item $G(\vv, \ee)$ is a directed graph with the  set $\vv$  of nodes (or vertices) and the  set $\ee$ of links (or edges). A link $e\in \ee$  from node $u$ to node $v$ is denoted  by $(u,v)$, where $u,v\in \vv$.
	\item $\dd$  and $\ww$ are mappings from $\ee$ to non-negative real numbers. For  each edge $e\in \ee$, function $q$ assigns $e$ with a   {\em capacity} $q(e)$,  and function  $\ww$ assigns $e$ with a weight $\ww(e)$, respectively.
\end{itemize}  

A commodity is a measure of the demand in a network. Formally, for  capacitated and weighted network $\mathcal{N}$, a \emph{commodity} (or demand) is a triple  $D = (s,t, d)$, where  $s$ and $t$ are nodes  of $\mathcal{N}$, and $d$ the bandwidth of  non-negative value. The nodes $s$ and $t$  are called \emph{source}  and \emph{ target} of commodity $D$, respectively. We are now ready to formulate the basic model of {\bf MCF} below.

\begin{model}[\MCF \label{model:base}]
Given a capacitated and weighed network $\mathcal {N}$, let $\KK=\{D_1,D_2,\cdots, D_{\LL}\}$ be a set of $\LL$ commodities, where $D_i=(s_i,t_i,d_i)$ on $\mathcal{N}$, and $f_{i}(u,v)$  be a variable for each link $(u,v)$ of  $\mathcal{N}$ that takes values in the interval $[0,d_i]$, for $i=1,\cdots, \LL$.  The {\em basic multi-commodity flow problem} is to solve the following linear equation for the flow variables $f_i(u,v)$  with   four constraints:

\begin{align}
	\min\quad& \sum_{(u,v)\in \ee} \left( \ww(u,v) \sum_{i=1}^{\LL} f_i(u,v) \right) \label{opt:cost}\\
	\mbox{s.\ t.}\quad &\sum_{i=1}^{\LL}  f_i(u,v) \le \dd(u,v) \quad   (u,v)\in \ee,    \quad \quad \quad\   \quad \quad \quad \quad \quad \quad (\mbox{Capacity constaints}) \label{cons:cap}\\
	&\sum_{v\in \vv }  f_i(u,v)= \sum_{v\in \vv }  f_i(v,u),\quad   u\in \vv\setminus \{  s_i, t_i\},\ i=1,\cdots,\LL\quad (\mbox{Flow conservation})  \label{cons:cons} \\
	&\sum_{v \in \vv} f_i(s_i, v) =\sum_{v\in \vv} f_i(v,t_i) = d_i,\quad    i=1,\cdots, \LL \qquad  \qquad  (\mbox{Demand satisfication})  \label{cons:dem}\\
	&f_i(u,v)\ge 0,  \quad  i=1,\cdots,\LL \mbox{ and } (u,v)\in \ee  \nonumber 
	\end{align}
\end{model}
Notice that constraint (\ref{opt:cost}) is an objective function.   The basic \MCF\ formulation, the flow variables $f_i(u,v)$ of the commodities of $\KK$ represents the  fraction of flow for commodity  $D_i$ along edge $(u,v)$. Thus, $f_i(u,v)\in [0,d_i]$ in the general case when the commodity $d_i$ can be split among the flows of multiple paths, and $f_i(u,v)$ can only take one of the two possible valued  $\{0,d_i\}$ otherwise (i.e. ``single path routing"). In this paper, we focus   on $f_i(u,v)\in[0,d_i]$. Taking the capacities and weights $\dd(u,v)$ and $\ww(u,v)$  of the edges  $(u,v)\in \ee$ as the cost element,  finding an  assignment  $f=(f_1, \cdots, f_\LL)$ in the above linear equation problem is called the  {\em minimum cost multi-commodity flow problem} ({\bf min-MCF}), indicated by constraint (\ref{opt:cost}).

\subsection{Node-Link Formulation}
\label{sec:nodelink}
In {\bf Model 1}, constraint~(\ref{cons:dem}) requires that the  demand $d_i$ of each  commodity is fully  delivered through the flows along the paths from the source to the target. However, in general, only a part of the demand of a commodity can be ``successfully''  delivered, which means that constraints~(\ref{cons:dem}) become 

\[\sum_{v \in \vv} f_i(s_i, v) =\sum_{v\in \vv} f_i(v,t_i) \le d_i\]
where $i=1,\cdots, \LL$.

\begin{figure}[h!]
	\centering{
		\includegraphics[width=0.7\textwidth]{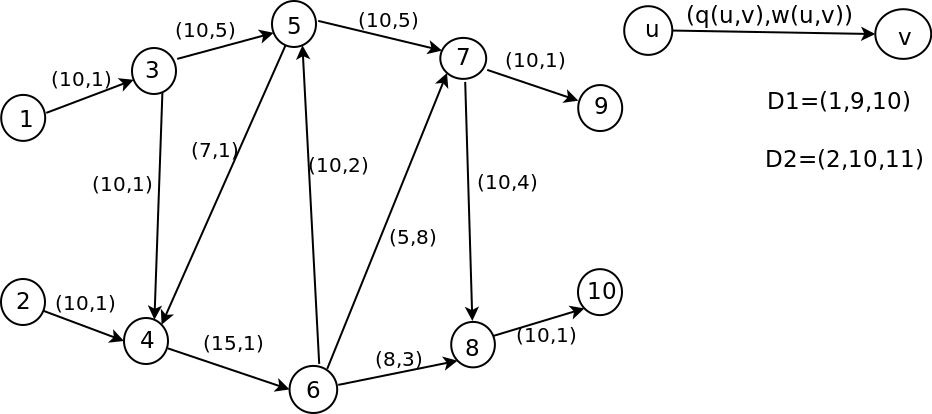}
		\caption{A \MCF\ example.\label{example:1}}
	}
	\begin{flushleft}
		{The network is given as above  figure. There are $10$ nodes, $13$ edges:
			\[\ee=[(1,3),(3,5),(5,7),(7,9),(2,4),(4,6),(6,8),(8,10),(3,4),(5,4),(6,5),(6,7),(7,9)]\]  and two commodities $D_1=(1,9,10),D_2=(2,10,11)$. The values of the pair of numbers on an edge from $u$ to  $v$ are the capacity  $\dd(u,v)$ and weight $\ww(u,v)$ of the edge, respectively. }
	\end{flushleft}
\end{figure}

Then it is desirable to seek the maximum portion of the command of each commodity to be successfully delivered with  minimum cost.  This  case of {\bf MCF} is  called the  \emph{maximal  multi-commodity flow problem} (\MMCF).  The primary requirement of  {\bf MMC} is  to try to deliver  all the demand, and the secondary requirement is to  minimize the total cost. 

 We use $|S|$ to denote the cardinality of  set $S$, and  $|\AA|$ to denote the dimension of a  square matrix $\AA$.

\begin{model}[Node-Link Formulation \cite{Jajszczyk2005}]
	\label{model:node}
The formal description of \MMCF\  is defined  as follows:
\begin{align*}
  \MIN\quad &   \sum_{(u,v)\in \ee} \left( \ww(u,v) \sum_{i=1}^{\LL} f_i(u,v)\right) +&W \sum_{i=1}^{\LL}\left(d_i-
  \sum_{v \in \vv} f_i(s_i, v)  \right)   \\
  \mbox{s.\ t.}\quad &\sum_{i=1}^{\LL}  f_i(u,v) \le \dd(u,v),     & (u,v)\in \ee  \\
  &\sum_{v\in \vv }  f_i(u,v)=\sum_{v\in \vv }  f_i(v,u),        &  u \in \vv \setminus \{  s_i, t_i\} \mbox{ for } i=1,\cdots, \LL 
  \\
  &\sum_{v \in \vv} f_i(s_i, v) =\sum_{v\in \vv} f_i(v,t_i) \le d_i, &  i=1,\cdots, \LL \\ 
    &f_i(u,v)\ge 0,                                                  &  i=1,\cdots,\LL\mbox{ and }  (u,v)\in \ee 
\end{align*} 
where  $W$ is a nonnegative real number that satisfies $W> \max\{\PW_{p}\mid  p\in \PP_i, \mbox{ for } i=1\cdots, \LL \} $ and  $\PW_{p}=\sum_{(u,v) \in p} {\ww(u,v)}$. 

\end{model}

$W \sum_{i=1}^{\LL}\left(d_i-
\sum_{v \in \vv} f_i(s_i, v)  \right)$is the \emph{penalty term} in the objective function.  \emph{Node-Link Formulation} is   a 
linear programming model with $\LL|\ee|$ variables and
$|\ee|+\LL(|\vv|-1)$  constraints.  It is easy to see that  \MCF\ is a special case of \MMCF\  when $\sum_{i=1}^{\LL}\left(d_i-
\sum_{v \in \vv} f_i(s_i, v)  \right)=0$, which means that  all commodities are successfully delivered.

\begin{example}
	In Fig. \ref{example:1}, we can choose $W$ as sum of all links' weight, which is $34$.
\end{example}

\subsection{Link-Path Formulation}
\label{sec:linkpath}

In the previous two models of linear equations, the variables are the accounts of flows of links. We now present a formulation based on the accounts of flows of paths. For a path $p$, we denote the account of flow along  path $p$  as  a  variable $\xx_{p}$.  For an arbitrary path $p$ and an edge $e$, we define the following (characteristic) function

\[ \proj_{p,e} =
\begin{cases}
1       & \quad   \text{ if link } e \text{ belongs to path } p\\
0  & \quad \text{ otherwise.}
\end{cases}
\]
For a precise formulation of \MMCF,  we introduce the following notations below for a given set  $\KK$ of commodity.
\begin{itemize}
	\item Let $\PP_i$ denote an enumeration of the set of paths from $s_i$ to $t_i$ without loops (called {\em simple paths}), for $D_i=(s_i,t_i,d_i)$ and $i=1,\cdots, \LL$.
	\item Given a path $p$, let $(u,v)\in p$ denote that edge $(u,v)$  is in path $p$ and  path is along edge $(u,v)$.  
\end{itemize}

\begin{model}[Link-Path Formulation \cite{Jajszczyk2005}]\label{model:link}
 \MMCF\ can be described as a problem of finding an assignment to the variables $\xx_{p}$ for $p \in \PP_i,\  i=1\cdots, \LL$,  satisfying the following constraints.
	\begin{align}
	\MIN\quad &   \sum_{i=1}^{\LL} \sum_{p\in \PP_i} \PW_{p}\xx_{p}+ W  \sum_{i=1}^{\LL} y_i & \label{model:final}	 \\
	\mbox{s.\ t.}\quad & \sum_{p\in \PP_i} \xx_{p}+y_i = d_i,& i=1,\cdots,\LL & \label{model:commodity}\\
	&\sum_{i=1}^{\LL}\sum_{p\in \PP_i} \xx_{p} \proj_{p,e} \le \dd(e), &  e\in \ee  & \label{model:edge} \\
	&\xx_{p} \ge 0,y_i\ge 0,&   p\in \PP_i,\ i=1,\cdots, \LL& \nonumber 
	\end{align} 
\end{model}
In this model, (\ref{model:final}) is the objective function, $y_i$ are \emph{slack variables} which represent the  portion of  demand for commodity $D_i$ that fails to be delivered, and  $ W  \sum_{i=1}^{\LL} y_i $ is the \emph{penalty term} to  objective function.
	  \emph{Link-Path Formulation} is a linear programming model  with  $\LL+|\ee|$ constraints and $\sum_{i=1}^{\LL} |\PP_i| +\LL$ variables. 
	It is easy to see that $\sum_{i=1}^{\LL} |\PP_i| +\LL$  might become very large even for a small network.
	
\begin{example}
	Fig. \ref{fig:ans} is the topology (a bidirectional graph) of one backbone network of USA. This topology has 18 nodes and  52 
	links. Even in this small topology there are 97 different simple paths that connect {\bf Hawaii} and {\bf Hartford}.
	\begin{figure}[ht]
		\centering
		\includegraphics[width=0.8\textwidth]{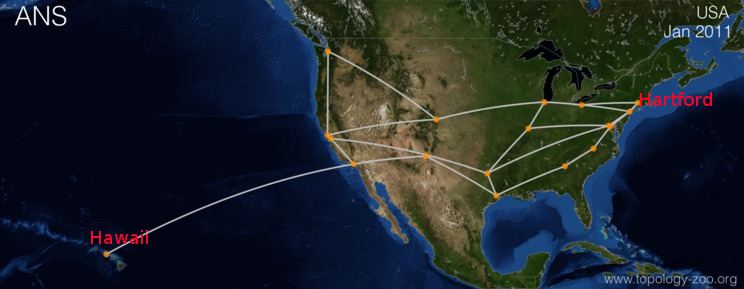}
		\caption[Caption without FN]{Network topology of USA.\label{fig:ans}}
		The original picture can been found in \url{http://www.topology-zoo.org/maps/Ans.jpg}
	\end{figure}	
\end{example}
In summary, we can see 
\begin{enumerate}
 \item both  Node-Link Formulation and Link-Path Formulation are linear programming model.
 \item Node-Link Formulation has fewer variables  than Link-Path Formulation, while Link-Path Formulation has fewer
constraints than Node-Link. 
 \item In general, both models either  too  many variables or too many constraints in practice.
\end{enumerate}


\section{The Column Generation Algorithm for Multi-Commodity Flow Problem}
\label{sec:review}
In this section, we first review the  classical  column generation. We then introduce a transition system model for understanding and analysis of this algorithm and  the improved algorithm that we propose later.  Finally in this section, we present the  matrix formulation of classical  column generation. 
\subsection{ The Algorithm of Column Generation}
\label{sec:class}

The variables in  {\bf Model \ref{model:link}} are often too many  to be dealt with explicitly. Luckily, column generation~\cite{ford2004a}
treats non-basic variables implicitly. It replaces the traditional method for determining a vector to entering
 basic by finding a shortest path which connects commodity source and target.
It has   better performance  than the simplex method for Link-Path Formulation   because both the number of variables and constraints  are  reduced to
$ |\ee| +\LL$ during every iteration. The basic idea of the algorithm is as follows.

In order to design an algorithm for  full deliver of  each demand $d_i$, we introduce a dummy path for  for each commodity $D_i$, denoted by $dummy_i$. Let the  capacity  of $dummy_i$ be $d_i$ and  $\PW_{dummy_i}=W$,  where $W$ is value defined in {\bf Model \ref{model:node}}. We call the original network extended with the dummy paths  $dummy_i, i=1,\cdots,\LL,$,  the  {\em augmenting network}, and define $\mathcal{P}= \bigcup_{i=1}^{\LL}\{P_i \cup \{dummy_i\}\}$ to denote the set of all paths of the commodities of the augmenting network.

 The algorithm  iteratively 
updates the load flow $\xx_{p}$ for every path $p\in \mathcal{P}$, where $y_i  =\xx_{dummy_i}$  for the variables $y_i$ in {\bf Model \ref{model:link}},  $i=1,\cdots, \LL$.
.   When it terminates,  the values of path load  flows $\xx_{p}$ for all $p\in \mathcal{P}$ give an  optimal solution  for linear programming problem in  {\bf Model \ref{model:link}}.

\begin{mydef}
	Let  $e^{\mbox{th}}$  is the index of edge $e$ in $\ee$. We introduce
  edge $e$'s basic vector $\bb_e$  for $e\in \ee$,
  as follows: 
   \[\begin{array}{lcl}
   \bb_e[j]&=& 
   \begin{cases}
   1 & \mbox{ if } j=\LL+e^{\mbox{th}} \\
   0 & \mbox{ otherwise}
   \end{cases}
   \end{array} 
   \]
   $e^{\mbox{th}}$ is the index of edge $e$ in $\ee$.  In addition, we introduce  path $p$'s basic vector $\bb_p$
  for $p\in \PP_i, i=1,\cdots, \LL$, as follows:
  \[\begin{array}{lcl}
  \bb_p[j]&=& 
  	\begin{cases}
  	1 & \mbox{ if } j=i\\
	  	1 & \mbox{ if } j=\LL+e^{\mbox{th}}\mbox{ and }  \proj_{p_i,e}=1 \\
  	0 & \mbox{ otherwise}
  	\end{cases}
   \end{array} 
  	\]
     We define  $\bb_{dummy_i}$ as follows:
       \[\begin{array}{lcl}
     \bb_{dummy_i}[j]&=& 
     \begin{cases}
     1 & \mbox{ if } j=i\\
     0 & \mbox{ otherwise}
     \end{cases}
     \end{array} 
     \]
\end{mydef}

\begin{example}
	In Fig. \ref{example:1}, let $e=(3,5)$, then $\bb_e=[0,0,0,1,0,0,0,0,0,0,0,0,0,0,0]$. Let $p={1\rightarrow 3\rightarrow 5\rightarrow 7 \rightarrow 9}\in \PP_1$, then $\bb_p=[1,0,1,1,1,1,0,0,0,0,0,0,0,0,0]$.
\end{example}

For an assignment  $\xx_p$ of $p\in \PP_i$ and $i=1,\cdots,\LL$, the value   $\min\{ \dd(e) -\sum_{i=1}^{\LL}
\sum_{p_1\in \PP_i} \proj_{p_1,e}\xx_{p_1}\mid \proj_{p,e}=1\}$ is called the   {\em remaining capacity} of $p$, denoted by $\textit{RemainCapacity}(p)$. 
We say that $p$'s {\em remaining capacity  carries commodity}  $D_i$  if and only if $p\in \PP_i$ and $d_i$ is less or equal to $\textit{RemainCapacity}(p)$.

\begin{algorithm}
  \SetAlgoCaptionSeparator{.}
  \DontPrintSemicolon
  \KwIn{ $G(\vv,\ee), \dd,\ 　\ww,\ D_i=(s_i,t_i,d_i)$ for $i=1,\cdots, \LL$  }
  \KwOut{$\{\xx_{p}\mid p\in \mathcal{P}\}$  an optimal solution for {\bf Model \ref{model:link}}.}
  \tcc{Initial solution.}
 
  \For {$i \in \{1,\cdots,\LL\}$}
	   {
 \label{init:start}
		 Compute a shortest path $p_i$ in $G(\vv,\ee)$  such that $\textit{RemainCapacity}( p_i)$  carries   $D_i$;\;
		 \lIf{$p_i$ exists}{set $\xx_{p_i}=d_i$; \tcc*{Reduce remaining capacity of edge $e$ remaining capacity of an edge not defined by $d_i$ where $e$ belongs to $p_i$; }}
		 \lElse{set $p_i=dummy_i,\ \xx_{dummy_i}=d_i$;}	
	   }
	    \label{init:end}
	   $k=1;\delta_i=0,i=1,\cdots,\LL; \ww'=\ww$;\tcc*[f]{Set temperary link weight vairable.}\;
	   $\uu[e]= 0, e\in \ee$;\tcc*[f]{Dual value for link.}\;

	   \tcc{Main iteration procedure.}
	   \For{$k>0$}
		     {\label{itertation:start}
		 	   \tcc{Choose entering variable.}
			   Let $e^*=\argmin_{e\in \ee}{\uu[e]}$;\;
			   \lIf{$\uu[e^*]<0$ \tcc*{Entering variable is a link $e^*$.}}
			       { $e^*$ is the entering variable;} 
			       \Else{  \tcc*{Entering variable is a path.}
				     \For {$i \in \{1,\cdots,\LL\}$}
				          {
				  	        Compute a shortest path $p_i'$ from $s_i$ to $t_i$ by  weight $\ww'$;\;
				  	        Let $\delta_i=\PPW_{p_i'}-\PPW_{p_i}$;\;
				          }
				          Let $j=\argmin_{i=1}^{\LL}{\delta_i}$;
				          \lIf{$\delta_j>=0$}{\Return $\{\xx_{p}\mid p\in \mathcal{P}\}$ }
				          \lElse{
				  	        $p_j'$ is the entering variable;
				          }
				   }
		 		   $\bb$ is the basic vector for the entering variable;\; 
		 		   \tcc{Choose leaving variable. The leaving variable is primary path $p_{k,i}$ or path $q\in \QQ_{k,i}$ or a link $e$. }
		 		   Apply classical pivot rule in 
		 		   \emph{simplex method} to choose leaving variable; \;
		 		   Update basic matrix;\;
			 	   Compute dual values $\uu$ and update new link weight map $\ww'=\ww+\uu$;	\;
			 	   $k=k+1$;\;		
			 	   \label{itertation:end}	 
		     }
	         
		     \caption{\SCG\label{alg:standard}}
\end{algorithm}

\subsection{Transition System Model}\label{subsec:trans}
To help the understanding and analysis of Algorithm \SCG, we introduce a state transition system that models the state change by each iteration of the main loop of the algorithm, i.e. {\bf lines \ref{itertation:start} - \ref{itertation:end}} of Algorithm \SCG. To define the abstract states of the transition system, we need the invariant property of the algorithm in the following lemma. 

\begin{lemma}
	Constraint (\ref{model:commodity}) in {\bf Model \ref{model:link}} is an invariant  of the main loop in  Algorithm \SCG\label{lemm:1} ({\bf lines \ref{itertation:start} - \ref{itertation:end}}).
\end{lemma}
\begin{proof}
The lemma holds because of the fact that the values of the variables $\{\xx_p~|~ p\in \mathcal{P}\}$ are alway kept in their feasible area is an invariant of the simplex method.\qed
\end{proof}
Since $d_i>0 $, {\bf  Lemma \ref{lemm:1}} implies that for each iteration, say the $k$th iteration, 
there is at least one $p\in (\PP_i\cup\{dummy_i\})$  for $i=1,\cdots,\LL$ such that
$\xx_p> 0$.  For the  $k$th iteration and  commodity $D_i$, a  path $p_{k,i} \in \PP_i\cup\{dummy_i\}$ which  has positive flow can be selected as the {\em primary path}  and the subset $\QQ_{k,i}\subseteq \PP_i\setminus\{p_{k,i}\}$ of paths which have non-zero flow as the {\em secondary paths} of $D_i$,   where $k, i=1, \ldots, l$. 

We now describe the main loop of Algorithm \SCG\ as  the transition system such that the $k$th iteration  changes from a state of the form  $([(p_{k,1}, \QQ_{k,1}),\cdots, (p_{k,\LL},\QQ_{k,\LL})],\NN_{k})$  to a state $([(p_{k+1,1}, \QQ_{k+1,1}),\cdots, (p_{k+1,\LL},\QQ_{k+1,\LL})],\NN_{k+1})$ 
 where  $ \NN_{k}, \NN_{k+1}\subseteq  \ee$. 

After initial solution steps in Algorithm \SCG\  (line \ref{init:start} to line \ref{init:end}), the system state is $p_{1,i}=p_i,\ \QQ_{1,i}=\emptyset$  for $i=1,\cdots, \LL$ and $\NN_1=\ee$.
The transition rules are defined in the following way. 

\begin{framed}
	
  \begin{enumerate}
		
  \item When the entering variable is a link $e^*$  :
	\begin{enumerate}
	\item When the leaving variable is a path $p_{k,i}$:
	  
	  By {\bf Lemma \ref{lemm:1}}, there is a $p\in \QQ_{k,i}$. Let $p_{k+1,i}=p,$ $\QQ_{k+1,i}=\QQ_{k,i}\setminus\{p \}$ and $ \NN_{k+1}=\NN_k\cup\{e^*\}$, the other $(k+1)$th's 
	  state are the  same as $k$th's state. In the following description,  without loss of generality we do not mention the unchanged state part. 
	\item When the leaving variable is a path $p\in \QQ_{k,i}$:
	  
	  Let $\QQ_{k+1,i}=\QQ_{k,i}\setminus\{p \},$ $ \NN_{k+1}=\NN_k\cup\{e^*\}$.
	\item When the leaving variable is a link $e$:
	  
	  Let $ \NN_{k+1}=(\NN_k\cup\{e^*\})\setminus\{e\}$.

	\end{enumerate} 
  \item When the entering variable is a path $p'_j$:
	\begin{enumerate}
	\item When the leaving variable is a path $p_{k,i}$:
	  \begin{enumerate}
	  \item When $i=j$:
		
		Let $p_{k+1,i}=p'_j$.
	  \item When $i\neq j$:
		
		By {\bf Lemma \ref{lemm:1}}, there is a $p\in \QQ_{k,i}$. \\   Let $p_{k+1,i}=p,$ $\QQ_{,k+1,i}=\QQ_{k,i}\setminus\{p \},$ $\QQ_{k+1,j}=\QQ_{k,j}\cup\{p_j'\}$. 

	  \end{enumerate}
	\item When the leaving variable is a path $p\in \QQ_{k,i}$:
	  \begin{enumerate}
	  \item When $i=j$:
		
		Let	 $\QQ_{k+1,i}=(\QQ_{k,i}\cup \{p_j'\})\setminus\{p \} $. 
	  \item When $i\neq j$:
		
		Let $\QQ_{k+1,j}=\QQ_{k,j}\cup\{p_j'\},$ $ \QQ_{k+1,i}=\QQ_{k,i}\setminus\{p\}$.
	  \end{enumerate}	
	\item When the leaving variable is a link $e$:
	  
	  Let $\QQ_{k+1,j}=\QQ_{k,j}\cup \{ p_j'\}, $ $ \NN_{k+1}=\NN_k\setminus\{ e\}$.

	\end{enumerate}
  \end{enumerate}
  \captionof{figure}{Transition system rules  for   \SCG. \label{sys:trans}}
\end{framed}
It is easy to see that state $([(p_{k,1},\QQ_{k,1}),\cdots, (p_{k,\LL},\QQ_{k,\LL})],\NN_k)$ represents the basic  matrix $\AA_k$ in the $k$th iteration,
where
\begin{equation} 
\AA_k=[\bb_{p_{k,1}}, \underbrace{\bb_{p}}_{p\in \QQ_{k,1}}, \cdots, \bb_{p_{k,\LL}},\underbrace{\bb_{p}}_{p\in \QQ_{k,\LL}},\underbrace{\bb_e}_{e\in \NN_k}]\label{eq:matrix}
\end{equation}
In other words, $\AA_k$ is the incidence matrix of paths  $p_{k,1}, \QQ_{k,1},\cdots, p_{k,\LL},\QQ_{k,\LL}$ and edges in $ \NN_k$

\begin{mydef}

In the above transition system, if the entering variable is a path $p$ and the leaving variable is a link $e$, then we call $p$   a basic variable which corresponds to $e$, denoted as $p_e$.

\end{mydef}
The variable $p_e$ has some update rules.
 When the entering variable is a path $p$ and the leaving variable is a path $p_e$, then we update $p_e=p$.  If the entering variable is a link $e_1$ and leaving variable is a path $p_e$, then update $p_e=p_{e_1}$.

\begin{note}
Let	$\SS_k=\ee  \setminus \NN_k,\  \QQ_k=\bigcup_{i=1}^{\LL} \QQ_{k,i}. $  We call $\SS_k$ is the set of   saturated link.
	\label{note:s}
\end{note}
 The intuitive meaning of a  saturated link is that its  bandwidth has been fully taken up and its bandwidth restricts the objective function to further decrease under current basis.

\begin{lemma}
	$|\QQ_k|=|\SS_k|$ is an invariant of the main loop, in other words,
 there  is a path $p_e\in \QQ_k$ for each $  e\in \SS_k$.
\end{lemma}
\begin{proof}
	We prove it by induction. When $k=0$, it  obviously holds. Assume that the conclusion holds when $k\le K_1$.
	
	When $k=K_1+1$,  if rules 1-(a)  and 1-(b) are used in Fig. \ref{sys:trans}, then both cardinal of $\QQ_k$ and $\SS_k$ decrease by $1$
	compared with last iteration. Hence conclusion holds. If rules  1-(c), 2-(a)-i and 2-(b) are used, then both  cardinal of $\QQ_k$ and $\SS_k$ are unchanged.  Hence conclusion holds. If rule 2-(a) is used, then  both  cardinal of $\QQ_k$ and $\SS_k$ increase by $1$.
	 Thus, conclusion holds. In summary, no matter what rule is used in $k$-th, 	$|\QQ_k|=|\SS_k|$  holds.
\qed
\end{proof}

\subsection{Matrix Formulation}\label{subsec:mat}

We  fix working paths on $p_1,\cdots, p_{\LL}, \QQ_1,\cdots, \QQ_{\LL} $ and  add  slack variables $z_e$ for constraint (\ref{model:edge}) where $e\in \NN_k$, then \MMCF\ can be described as follows:

\begin{model}[Link-Path Formulation for augmenting network\label{model:new}]
\begin{align}
  \MIN \quad & \sum_{i=1}^{\LL} \left(\ww_{p_i}\xx_{p_i} + \sum_{p\in \QQ_i}\ww_{p} x_{p}\right) & \label{matrix:opt} \\
  \mbox{s.\ t.}\quad  & \xx_{p_i}+\sum_{p\in \QQ_i} \xx_{p} = d_i, &  i =1,\cdots, \LL &\label{matrix:commodty} \\
    &\sum_{i=1}^{\LL}\left(\xx_{p_i}\proj_{p_i,e}+\sum_{p \in \QQ_i} \xx_{p}  \proj_{p,e} \right) = \dd(e),& e\in \SS_k & \label{matrix:edge1}\\
  &\sum_{i=1}^{\LL}\left(\xx_{p_i}\proj_{p_i,e}+\sum_{p \in \QQ_i} \xx_{p}  \proj_{p,e} \right)+z_e = \dd(e),& e\in \NN_k & \label{matrix:edge}\\
  &\xx_p\ge 0, &  p \in \bigcup_{i=1}^{\LL} \left(\QQ_i \cup \{p_i\}\right) & \label{matrix:var1}\\
  & z_e\ge 0, &  e\in  \NN_k   &              \label{matrix:var3}
\end{align}
\end{model}

\begin{note}
	$\cc_k=[\ww_{p_{k,1}},\underbrace{\ww_{p}}_{p\in \QQ_{k,1}},\cdots,\ww_{p_{k,\LL}}, \underbrace{\ww_{p}}_{p\in \QQ_{k,\LL}},
	\underbrace{0,\cdots, 0}_{|\NN_k|}].$		 
	$\rhs=[d_1,\cdots,d_{\LL},\underbrace{\dd(e)}_{e\in \ee}]$.
	
\end{note}

\begin{example}
	In Fig. \ref{example:1}, $\rhs=[10,11,10,10,10,10,10,15,8,10,10,7,10,5,10]$
\end{example}
In other words, {\bf Model \ref{model:new}} can be written as 
\begin{model}[Matrix formulation \label{model:matrix}]
	\begin{align*}
	\MIN\quad  &\cc_k\xx\\
	s.\ t. \quad &\AA_k\xx=\rhs.
	\end{align*}
	where $\AA_k$ is defined as (\ref{eq:matrix}).
\end{model}

In $k$th iteration the  leaving variable selection procedure in Algorithm \SCG\ 　 can been described as follows:
\begin{framed}	
\begin{empheq}[box=\mymath]{equation}
\AA_k \xx=\rhs \label{exit:flow} 
\end{empheq}

 Firstly, solve equation (\ref{exit:flow}), and obtain  \[\xx=[\xx_{p_{k,1}},\underbrace{\xx_{p}}_{p\in \QQ_{k,1}},
  \cdots, \xx_{p_{k,\LL}},\underbrace{\xx_{p}}_{p\in \QQ_{k,\LL}},\underbrace{z_e}_{e\in \NN_k}]\]  which satisfy constraints
(\ref{matrix:commodty})-(\ref{matrix:var3}).

Secondly, solve equation (\ref{exit:coef}), and obtain $\ll$.

\begin{empheq}[box=\mymath]{equation}
	  \AA_k \ll= \bb \label{exit:coef} 
\end{empheq}

where vector $\bb$ is the basic vector corresponding to entering variable.

Finally, choose a leaving variable by a \textit{pivot rule}. 

\end{framed}

\begin{note}[pivot rule]
  There are many different mays to choose leaving variable. In this paper, we apply classical pivot rule to pick $j$th  as leaving variable where

  \[j=\argmin_{i=1,\ll_i>0}^{\LL+|\ee|}\frac{\xx_i}{\ll_i}. \]
\end{note}

The solution $\uu$ of equation (\ref{exit:dual}) are dual values of   constraints (\ref{matrix:commodty}) and (\ref{matrix:edge}).	Then let
$\ww'=\ww+\uu$ be newly updated link weights.

\begin{empheq}[box=\mymath]{equation}
\left(\AA_k \right)^{\intercal}\uu =-\cc_k\label{exit:dual}	
  \end{empheq}

\subsection{Classical Column Generation Complexity Analysis}
\label{sec:stdana}

Suppose Algorithm \SCG\ does $h$ main iterations before termination, then Algorithm \SCG\
computes  $(h+1)\LL$ shortest path and  solves $3h$  linear equation systems in the form of  (\ref{exit:flow}), (\ref{exit:coef}) and (\ref{exit:dual}) where  $\AA_k$'s  size is $(\LL+|\ee|)\times (\LL+|\ee|)$.
As the authors know that the best shortest path  algorithm complexity is $O\left(|\ee| + |\vv| \log (|\vv|)\right)$ which is given by  Dijkstra's algorithm 
based  on Fibonacci heap and the best linear system solving algorithm complexity is $O\left( \left(\LL+|\ee| \right)^{2.376}\right)$. Hence, the Algorithm  \SCG's complexity is 
\begin{equation}
  O\left( h\left( \LL\left(|\ee| + |\vv| \log (|\vv|)\right) +  \left(\LL+|\ee| \right)^{2.376} \right) \right).\label{complex:1}
\end{equation}


\section{Speedup  Through Employing $\AA_k$'s Structure}
\label{sec:struct}

The complexity of (\ref{complex:1}) can not be accepted in reasonable time when the size of $G(\vv,\ee)$  is large. This hinders \SCG's  use
in some applications e.g.  online load balance  in {\bf SDN} and large scale problem  offline. Hence, how to improve the efficiency of column generation is a  problem considered in many works \cite{dantzig1967generalized,mccallum1977a,barnhart1994a,mamer2000a}.  The complexity  (\ref{complex:1}) only has two parts, i.e. computing shortest  path and solving linear equation systems (\ref{exit:flow}), (\ref{exit:coef}) and (\ref{exit:dual}). 　Hence, reducing coefficient matrix size is a feasible approach. Luckily,   the primal partitioning procedure, a specialization
of the generalized upper bounding procedure developed by Dantzig and
Van Slyke \cite{dantzig1967generalized}, involves the determination at each iteration of the inverse of a basis
containing only one row for each saturated link. In other words, we can reduce  matrix size to the number of saturated link.  In the following, we will concretely 
show how to apply conclusion of \cite{dantzig1967generalized}  on \MCF. Through
reordering column of  basis matrix    to obtain a special structure in resulted  basis matrix  $\AA_k$, we give  bellow a method  called \emph{structured matrix method} (\SMCG). By this way  we can reduce
the size of   linear equation to be solved   in general.

\subsection{Structured Matrix Method for Column Generation}

After we reorder basic variable in $k$th iteration by $p_{k,1},\cdots, p_{k,\LL}, \underbrace{p}_{p\in \QQ_{k,1}},\cdots, \underbrace{p}_{p\in \QQ_{k,\LL}},\underbrace{e}_{e\in \NN_k}$, 
{\bf Model~\ref{model:new}} can be rewritten as:


\begin{model}[Structure matrix model \label{model:structed}]
\begin{align*}
  \MIN\quad  &\cc_k\xx\\
 s.\ t.\quad &\AA_k\xx=\rhs_k
 \end{align*}
where
\begin{align}
  &\AA_k &=&
  \bbordermatrix{~  &\KK & \SS_k & \NN_k \cr
	\KK & \II& \BB_{k} & \zero\cr
	\SS_k& \CC_{k}  & \DD_{k} & \zero\cr
	\NN_k& \HH_{k}   & \FF_{k} & \II\cr
  } \label{eq:factor}\\
  &\cc_k &=& [\underbrace{ \ww_{p_i}}_{ i=1,\cdots,\LL}, \underbrace{\ww_{p_e}}_{e\in \SS_k} ,\underbrace{0,\cdots,0}_{e\in \NN_k}] =[{\cc}_{\KK}, {\cc}_{\SS_k},{\zero}_{\NN_k}] \nonumber\\
  &\rhs_k &=& [\underbrace{d_i}_{i=1,\cdots,\LL}, \underbrace{\dd(e)}_{e\in \SS_k}, \underbrace{\dd(e)}_{e\in \NN_k}]=[{\rhs}_{k,\KK},{\rhs}_{k,\SS_k},{\rhs}_{k,\NN_k}] \nonumber \\
  &\xx &=& [\underbrace{\xx_{p_i}}_{i=1,\cdots,\LL},  \underbrace{\xx_{p_e}}_{e\in \SS_k}, \underbrace{\xx_e}_{e\in \NN_k} ] = [\xx_{\KK}, \xx_{\SS_k}, \xx_{\NN_k}] \nonumber\\
  & \bb &=&[\bb_{\KK}, \bb_{\SS_k},\bb_{\NN_k}]\nonumber\\
   & \ll &=&[\ll_{\KK}, \ll_{\SS_k},\ll_{\NN_k}]\nonumber\\
  & \uu &=& [\uu_{\KK},\uu_{\SS_k},\uu_{\NN_k}] \nonumber\\
  & \II\ \mbox{ is an identity matrix} \nonumber \\
  &\xx\ge 0 \nonumber
\end{align}

\end{model}

 Mathematically, we can rewrite equation (\ref{exit:flow})  into:
\begin{align}
  &\xx_{\KK}+\BB_k\xx_{\SS_k}&=&\rhs_{k,\KK} \label{eqn:1}\\
  &\CC_k\xx_{\KK}+\DD_k \xx_{\SS_k}&=&\rhs_{k,\SS_k} \label{eqn:2}\\
  &\HH_k\xx_{\KK}+\FF_k\xx_{\SS_k} +\xx_{\NN_k} &=& \rhs_{k,\NN_k} \label{eqn:3}
\end{align} 
Then we have
\begin{equation}
  \CC_k\times (\ref{eqn:1})-(\ref{eqn:2})\implies \left(\CC_k \BB_k-\DD_k  \right)\xx_{\SS_k}=\CC_k\rhs_{k,\KK}-\rhs_{k,\SS_k}  \label{eq:simple1}
\end{equation}
 Hence, we can firstly solve equation (\ref{eq:simple1}). Secondly, substituting $\xx_{\SS_k}$ in (\ref{eqn:1}) to obtain  $\xx_{\KK}$. Finally, 
substituting $\xx_{\KK}, \xx_{\SS_k}$ 	in (\ref{eqn:3}) to obtain 
$\xx_{\NN_k}$. In this way can solve equation system (\ref{exit:flow}).

\begin{note}\label{note:g}
Let
\begin{empheq}[box=\mymath]{equation}
 \UU_k=\CC_k \BB_k-\DD_k  \label{eq:M}
 \end{empheq}

\end{note}
Now equation (\ref{eq:simple1}) can be written as:
\begin{empheq}[box=\mymath]{equation}
\UU_k\xx_{\SS_k}=\CC_k\rhs_{k,\KK}-\rhs_{k,\SS_k}  \label{eq:simple}
\end{empheq}

\begin{lemma}
	$\UU_k$ is a non-singular sparse matrix. \label{lemma:unsing}
\end{lemma} 
\begin{proof}
	By simplex method theory, $\AA_k$ is a non-singular 
	matrix. By structure of  $\AA_k$ in (\ref{eq:factor}), $\det\left(\AA_k\right)=\det\left(\UU_k\right)\neq 0$. Therefore,
	$\UU_k$ is a non-singular matrix.\qed
\end{proof}
	Since equations (\ref{exit:flow}) and (\ref{exit:coef}) have 
	the same coefficient matrix. Hence,
	employing the way to solve  equation (\ref{exit:flow}), we can  solve  equation (\ref{exit:coef}).  Through the same method we obtain
\begin{empheq}[box=\mymath]{equation}
		\UU_k\ll_{\SS_k}=\CC_k\bb_{\KK}-\bb_{\SS_k}  ,\label{eq:simplexcoef}
\end{empheq}
	 $\ll_{\KK} =\bb_{\KK}-\BB_k \ll_{\SS_k}  \mbox{ and }  \ll_{\NN_k}=\bb_{\NN_k}-\HH_k\ll_{\KK}-\FF_k\ll_{\SS_k}.$

In equation (\ref{exit:dual}),	mathematically, $\left(\AA_{k}\right )^T\uu= \cc_k$ can be rewritten as
	\begin{align}
	    &\uu_{\KK}+\CC_k^T\uu_{\SS_k} +\HH_k^T \uu_{\NN_k} &=&-\cc_{k,\KK} \label{duall:eq1}\\
		&\BB_k^T\uu_{\KK} +\DD_k^T \uu_{\SS_k}+\FF_k^T\uu_{\NN_k}&=&-\cc_{k,\SS_k}\label{dual:eq2}\\
		& \uu_{\NN_k}&=&\zero_{\NN_k}\nonumber
	\end{align}
We substitute $\uu_{\NN_k}=\zero_{\NN_k}$ in (\ref{duall:eq1}) and (\ref{dual:eq2}) to obtain 
	\begin{align}
	&\uu_{\KK}+\CC_k^T\uu_{\SS_k} &=&-\cc_{k,\KK} \label{duall:eq11}\\
	&\BB_k^T\uu_{\KK} +\DD_k^T \uu_{\SS_k}&=&-\cc_{k,\SS_k}\label{dual:eq22}.
	\end{align}
 We simplify equation systems (\ref{duall:eq11}) and (\ref{dual:eq22}) by
\begin{align}
\BB_k^T \times (\ref{duall:eq11})-(\ref{dual:eq22})&=\BB_k^T \uu_{\KK}+\BB_k^T \CC_k^T \uu_{\SS_k}- \BB_k^T \uu_{\KK}
-\DD_k^T\uu_{\SS_k} \nonumber \\ 
&\implies  \left(\BB^T \CC_k^T -\DD_k^T  \right) \uu_{\SS_k}= \cc_{k,\SS_k}-\BB_k^T\cc_{k,\KK}\nonumber \\
&\implies \nonumber
\end{align}
\begin{empheq}[box=\mymath]{equation}
\UU_k^T \uu_{\SS_k}  = \cc_{k,\SS_k}-\BB_k^T\cc_{k,\KK}
\label{dual:simple}
\end{empheq}

Through the above simplification, we can  solve equation (\ref{exit:dual}) too.

\subsection{Structure Matrix Method's  Complexity Analysis}

Suppose \SCG\ does $h$ main iterations before termination, then it
computes  $(h+1)\LL$ shortest path and in $k$th iteration we need  to solve $3$  linear equation systems in the form of (\ref{eq:simple}), (\ref{eq:simplexcoef}) and (\ref{dual:simple})  where  $\UU_k$'s  size is $|\SS_k|\times |\SS_k|$.
By the same discussion in Section \ref{sec:stdana}, we can obtain that 
the \emph{Structure Matrix Method} complexity is 
\begin{equation}
O\left( h\left( \LL\left( |\ee| + |\vv| \log (|\vv|)\right)  \right)+\sum_{k=1}^h |\SS_k|^{2.376}\right).\label{complex:stucture}
\end{equation}
As  given by the analysis in Section \ref{sec:stdana}, the standard column generation method is  
\[O\left( h\left( \LL\left(|\ee| + |\vv| \log (|\vv|)\right) +  \left(\LL+|\ee| \right)^{2.376} \right) \right)\] by (\ref{complex:1}). 
	By the  Note \ref{note:s},  it is easy to see that $|\SS_k|< \LL+|\ee|$. And  in general, $|\SS_k|< |\ee|$, hence \SMCG\  is better than the classical one (\SCG).

\section{Speedup Through Employing $\UU_k$'s Sparse Structure}
\label{section:sparse}
The sparse property is very useful when solving linear equation, in the following, to show the sparse property of matrix $\UU_k$ we will  discuss the element 
of matrix $\UU_k$ in detail.
 In $\UU_k=\CC_k \BB_k-\DD_k $,  $\CC_k[i]$ is a vector denoting  whether path $p_{k,i}$ 
 crosses each edge in $\SS_k$, i.e.
$$\CC_k[i][j] =\begin{cases}
 	1 \quad \text{ if }  jth \text{ edge of }  \SS_k  \text{ belongs to path   }p_{k,i},    \\
 	0 \quad  \text{ otherwise.} 
 \end{cases}$$
$\BB_k[i]$ is a vector  associated with $i$th  path of $[\QQ_{k,1},\cdots,\QQ_{k,\LL}]$ and its value indicates which commodity the  $i$th  path belong to. Let $i$th  path of $[\QQ_{k,1},\cdots,\QQ_{k,\LL}]$ be  path for
 commodity $D_h$. Then   
$$\BB_k[i][j] =\begin{cases}
1 \quad \text{ if } j=h, \\
0 \quad  \text{ otherwise.} 
\end{cases}$$ 
Hence, $\CC_k\BB_k[i]=\CC_k[h]$,  where $\CC_k[h]$ is associated with  path $p_{k,h}$ and 
$$\CC_k\BB_k[i][j] =\begin{cases}
1 \quad \text{ if } jth \text{ edge of }  \SS_k  \text{ belongs to path   }p_{k,h},    \\
0 \quad  \text{ otherwise.} 
\end{cases}$$

 $\DD_k[i]$　 is also a vector associated with  $i$th  path of $[\QQ_{k,1},\cdots,\QQ_{k,\LL}]$ and its value indicates
  which edge it crosses. Let path $p$ be the $i$th  path of $[\QQ_{k,1},\cdots,\QQ_{k,\LL}]$. Then,
   $$\DD_k[i][j] =\begin{cases}
	1 \quad \text{ if } jth \text{ edge of }  \SS_k  \text{ belongs to path   }p,    \\
	0 \quad  \text{ otherwise.} 
 \end{cases}$$

   By conclusion of \cite{fronczak2004average} the  ratio   between  path length and $|\ee|$  is very small
   when $|\ee|$ is large in general. When we see the graph as only consisted of saturated links, then number of nonzero elements in vector $\CC_k\BB_k[i]$ and $\DD_k[i]$  are identical with the   length of associated paths $p_{k,h}$ and $p$. Therefore,  $\CC_k\BB_k[i]$ and $\DD_k[i]$ are two sparse vectors.  As  discussed above,
  $\UU_k$ is statistically a very sparse matrix.
  
  In the following, we list some experiment results of matrix $\UU_k$. We record the dimension and number of  $\UU_k$'s nonzero elements  in every iteration for some random cases. Let $N(\UU_k)$ be the number of nonzero element in matrix $\UU_k$. The detail of
  cases' configuration can be found in section \ref{sec:exp}.  The dimension of $\UU_k$ is equal to number of saturated link in $k$th iteration. In Fig. \ref{fig:saturatelink} we can see that 
  the dimension starts from $0$ to a large number (more than $1000$), which indicates that the number of saturated links is more and more  larger as iteration proceeding, and the resource competition of different commodities 
   is more and more intense. But the  growth of the ratio between nonzero coefficients of  matrix $\UU_k$ and its dimension is very slow. In Fig. \ref{fig:mean1}  when  $k>500000$, the value $\frac{N(\UU_k)}{|\UU_k|}$ is still less than $5$ while $|\UU_k|$ is larger than $1000$.  Hence $\UU_k$ is a very sparse matrix.

  	\begin{figure*}
  		\centering
  		\begin{subfigure}[b]{0.475\textwidth}
  			\centering
  			\includegraphics[width=0.8\textwidth]{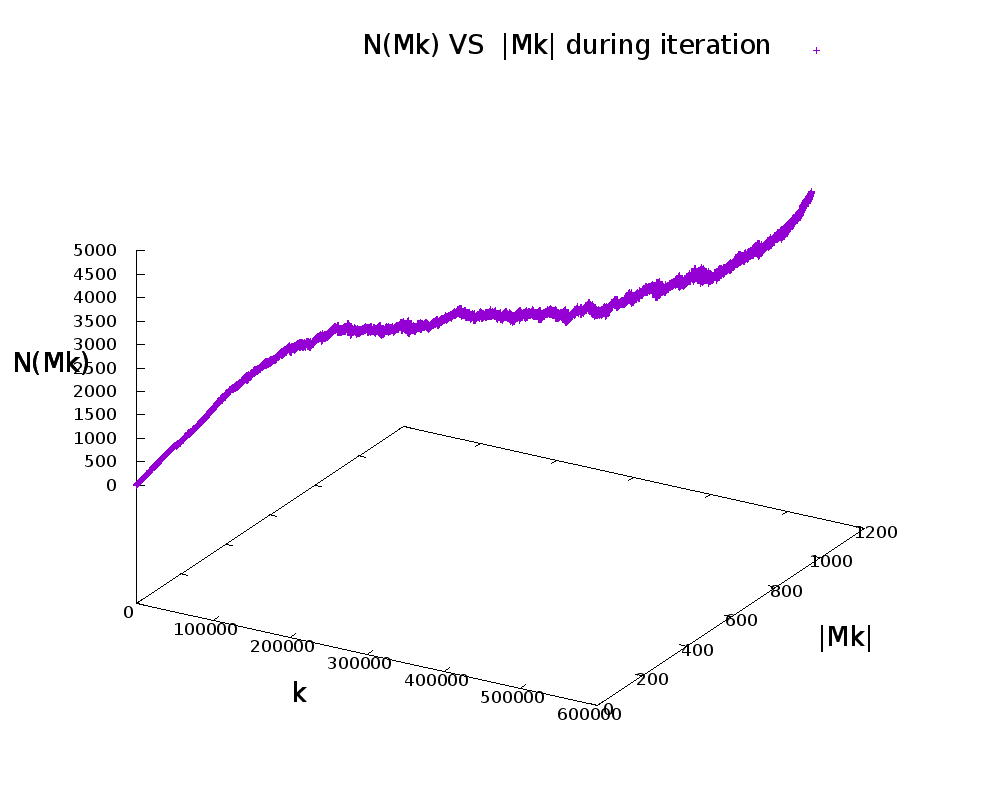}
  			\caption[Network2]%
  			{{\small $|\vv|=1000, |\ee|=5000, |\KK|=1000$}}    
  			\label{fig:mean1}
  		\end{subfigure}
  		\hfill
  		\begin{subfigure}[b]{0.475\textwidth}  
  			\centering 
  			\includegraphics[width=0.8\textwidth]{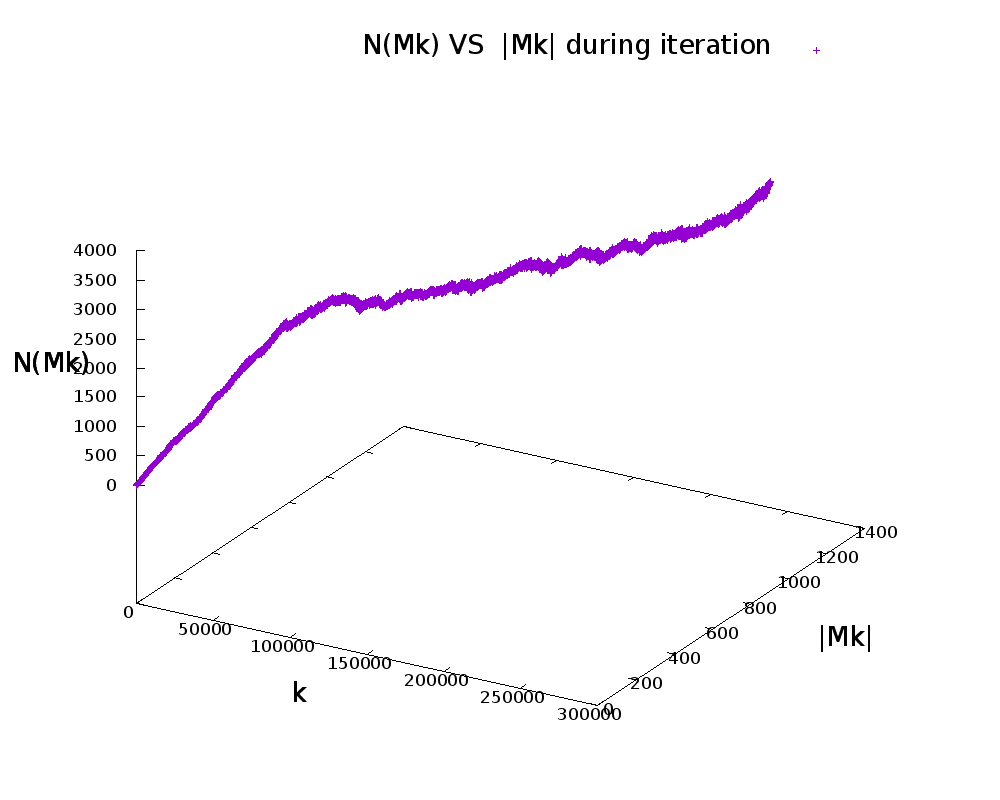}
  			\caption[]%
  			{{\small  $|\vv|=1500, |\ee|=7500, |\KK|=1000$}}    
  			\label{fig:mean3}
  		\end{subfigure}
  			\vskip\baselineskip
  			\begin{subfigure}[b]{0.475\textwidth}   
  				\centering 
  				\includegraphics[width=0.8\textwidth]{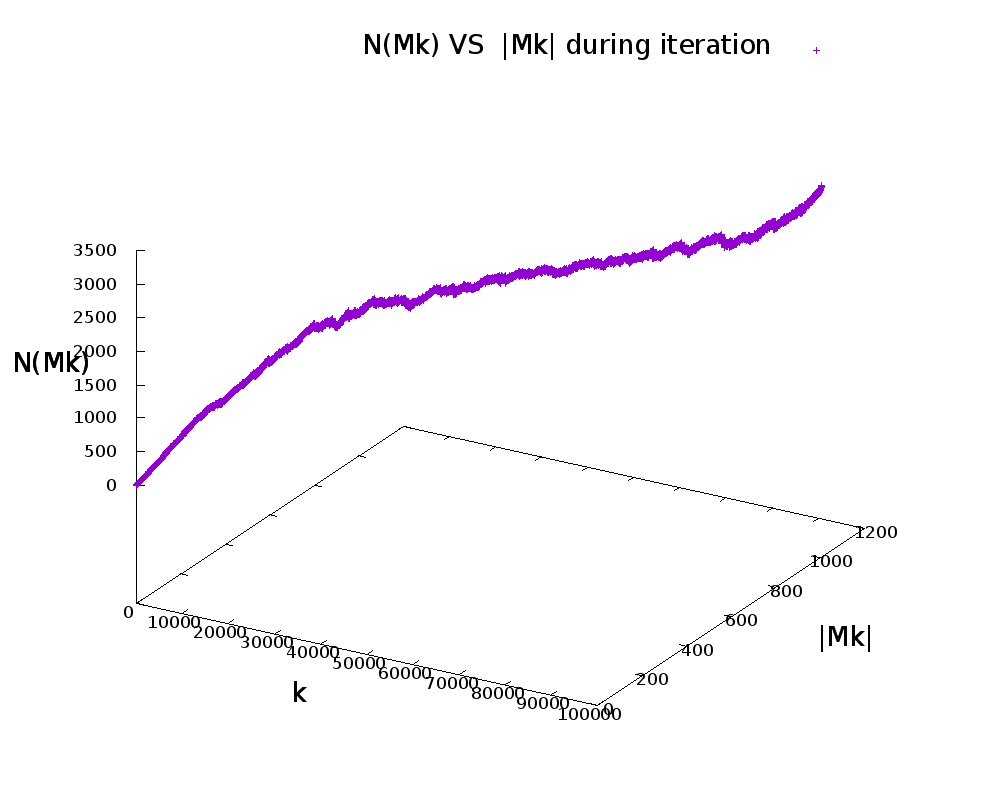}
  				\caption[]%
  				{{\small $|\vv|=2000, |\ee|=10000, |\KK|=1000$}}    
  				\label{fig:mean5}
  			\end{subfigure}
  			\quad
  			\begin{subfigure}[b]{0.475\textwidth}   
  				\centering 
  				\includegraphics[width=0.8\textwidth]{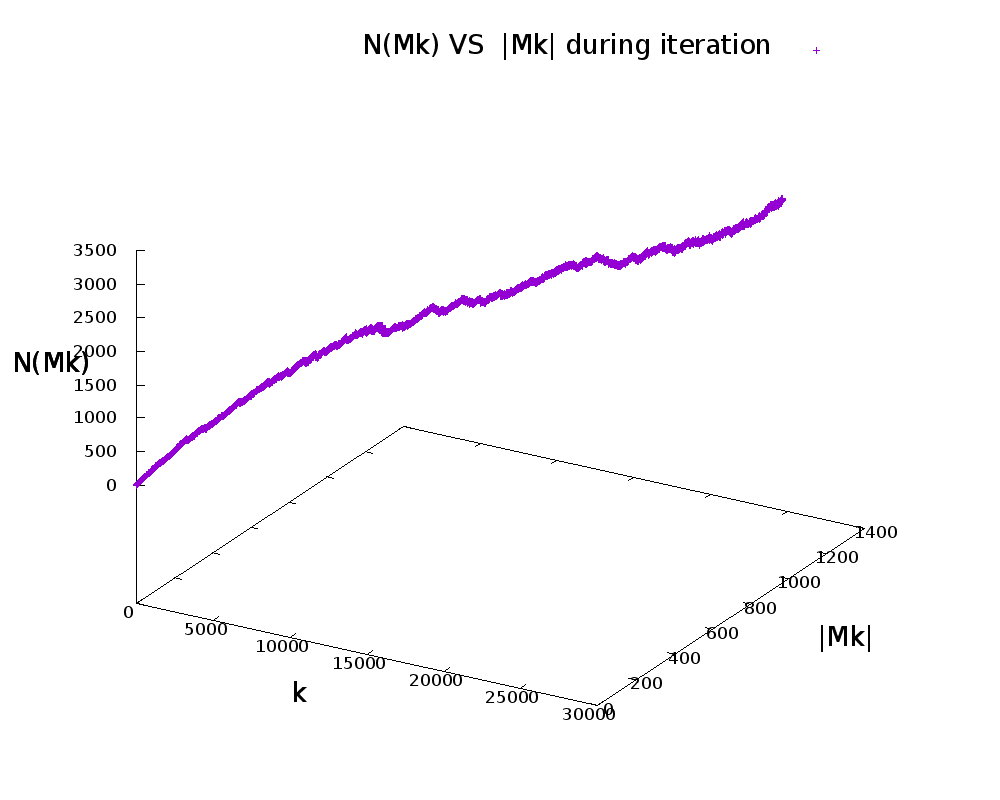}
  				\caption[]%
  				{{\small $|\vv|=3000, |\ee|=15000, |\KK|=1000$}}    
  				\label{fig:mean7}
  			\end{subfigure}
  		\vskip\baselineskip
  		\begin{subfigure}[b]{0.475\textwidth}   
  			\centering 
  			\includegraphics[width=0.8\textwidth]{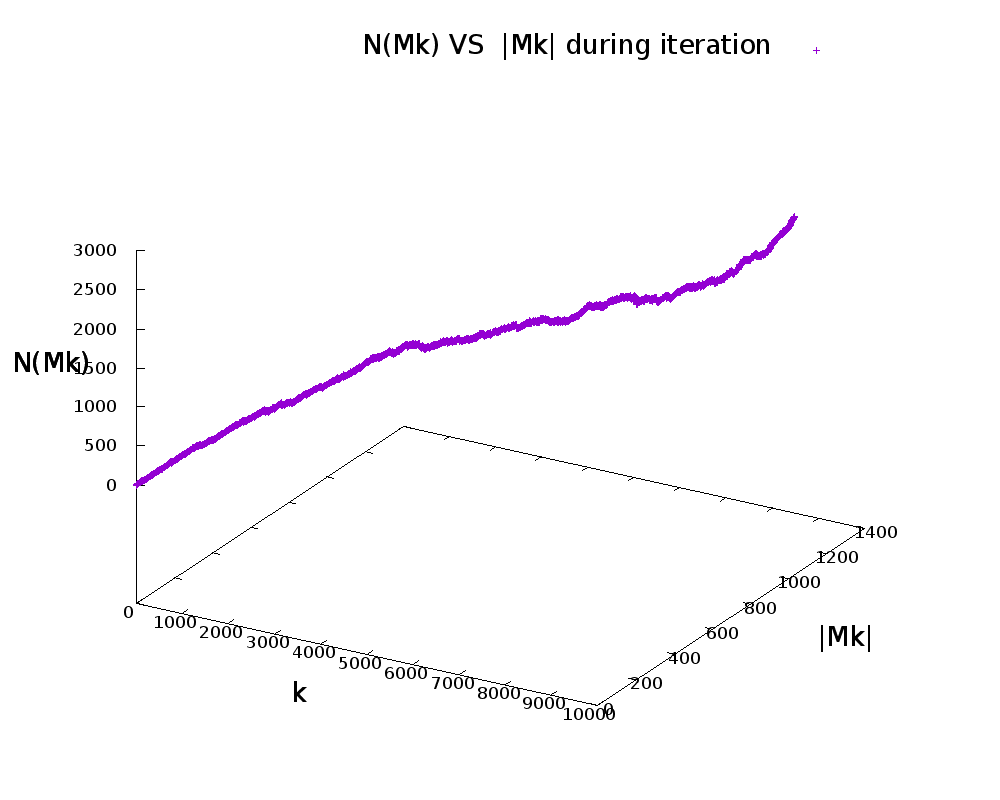}
  			\caption[]%
  			{{\small $|\vv|=5000, |\ee|=25000, |\KK|=1000$}}    
  			\label{fig:mean5}
  		\end{subfigure}
  		\quad
  		\begin{subfigure}[b]{0.475\textwidth}   
  			\centering 
  			\includegraphics[width=0.8\textwidth]{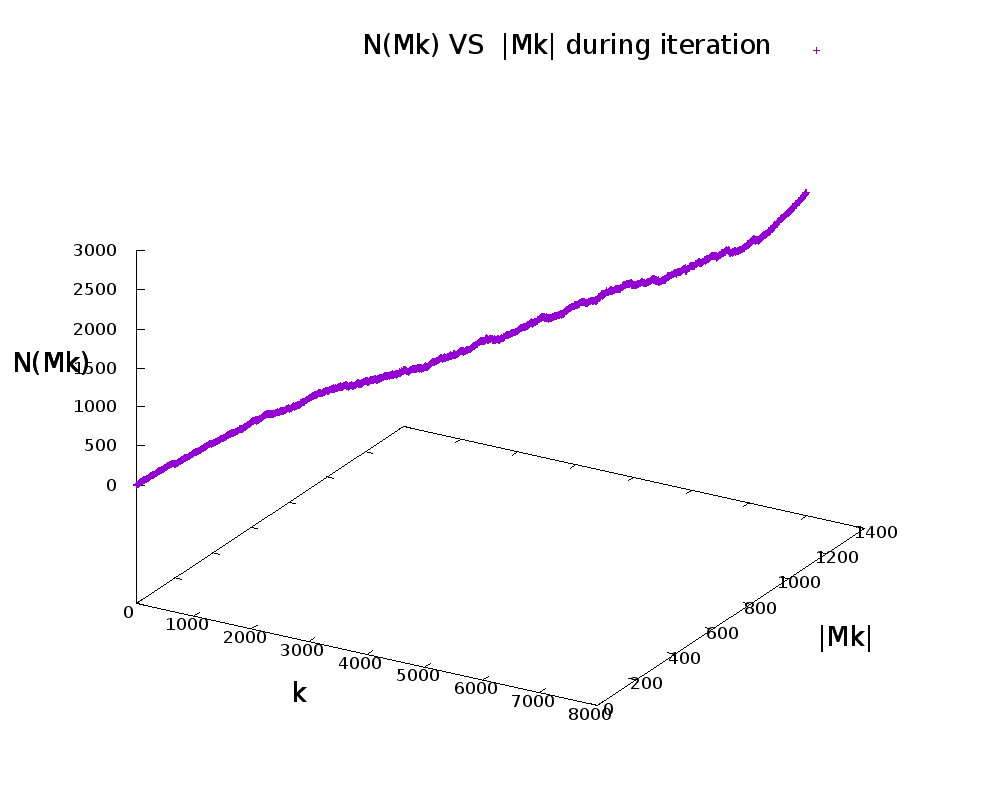}
  			\caption[]%
  			{{\small $|\vv|=7000, |\ee|=35000, |\KK|=1000$}}    
  			\label{fig:mean7}
  		\end{subfigure}
  		\caption[ The average and standard deviation of critical parameters ]
  		{\small The number of nonzero elements in matrix $\UU_k$ vs  its dimension during iteration} 
  		\label{fig:saturatelink}
  			\begin{flushleft}
  				{In above pictures, the average number of nonzero elements in  each row of matrix $\UU_k$ is less than $5$. Hence the matrices $\UU_k$ are very sparse.
  					 }
  			\end{flushleft}
  	\end{figure*}

According to \cite{vanderbei1998linear}'s suggestions, we use LU decomposition to solve equations (\ref{eq:simple}), (\ref{eq:simplexcoef}) and (\ref{dual:simple}). But because of the high sparsity of matrix $\UU_k$, \LAP\  \cite{anderson1990lapack}
kernels are not applicable.
Hence, we can use  the  linear solver \KLU   \ \cite{davis2010algorithm}, which has high performance  for  sparse matrix,   to solve  equations (\ref{eq:simple}), (\ref{eq:simplexcoef}) and (\ref{dual:simple}) instead of  \LAP .

In Table \ref{tab:compare} we find an interesting phenomenon  that when the value $\frac{N(\UU_k)}{|\UU_k|}$ is greater than or equal to   $3$ (case $R(1000)$), solving  linear equation will become dominating part of total 
time consumption. And when $\frac{N(\UU_k)}{|\UU_k|}$  is less than $3$, sparse linear solvers will greatly reduce the time of  linear equation solving.


\section{Speedup Through Sparse and Similar Properties}
\label{sec:advance}
After we  employ \KLU\ to solve  linear equations occurring in iteration of \SMCG, when the number of  saturated link is small,  then linear equation solving step is not a dominating part of time. But when   there are many saturated links the complexity  of \SMCG\  is almost the same as  classical one. When structure of 
matrix $\UU$ is complex (i.e.  nonzero elements of matrix $\UU_k$ is more than $3|\UU_k|$),  then linear equation solving step 
dominates the entire algorithm time,  even employing \KLU. 
Therefore, in the following,
we do  not invoke \KLU\ to solve equations 　but directly use  results in previous 
iteration to incrementally solve equations (\ref{exit:flow}), (\ref{exit:coef}) and (\ref{exit:dual}).

For keeping speedup, in this section we firstly give a fast method \locS\ which reduces {\bf Problem \ref{pro:1}} to a small one  in Section \ref{sec:sparse}. Secondly, we provide an incremental method \incS\ to  solve equations (\ref{exit:flow}), (\ref{exit:coef}) and (\ref{exit:dual}) in Section \ref{sec:incremental}. Thirdly,  in the final, we discuss why \locS\ and \incS\ are proper solvers for equation  during iteration in Section \ref{sec:qcg}.   

\subsection{A Fast Method to Solve Sparse Linear Equation System}
\label{sec:sparse}

\begin{problem}Solve linear equation system
	\[\AA\xx=\rhs\]
	where $\AA$ is a $n\times n$ matrix and $\rhs$ is a vector.
	\label{pro:1}
\end{problem}
We will provide a fast algorithm to solve {\bf Problem \ref{pro:1}}. This method can reduce linear equation system to a small one. Especially, when  $\AA$ is a very sparse matrix and $\rhs$ is also a sparse vector, this method is very powerful. For presenting this  fast method we first  give following definitions and lemmas.
\begin{mydef}
	For $n\times n$ matrix $\AA$, let $G(\AA)$ be the undirected graph of matrix $\AA$ with $2n$ nodes. $G(\AA)$ has a link $(i,n+j)$ iff $\AA[i,j]\neq 0$.    Let $\reach_{G(\AA)}(\HB)$
	be the set of nodes  reachable from   element of $\HB$ through $G(\AA)$.
\end{mydef}

\begin{lemma}
In {\bf Problem \ref{pro:1}}, let $\HB=\{i\mid \rhs[i]\neq 0 \}, \XI=\{i\mid i\in  \reach_{G(\AA)} (\HB), i< n\}$. If $\HB\not \subseteq \XI$, then {\bf Problem \ref{pro:1}} has no solution. \label{lem:sol}
\end{lemma}
\begin{proof}
Set $h\in \HB, h\not \in  \XI$.  Thus, all the elements in row $h$ are  zeros. Therefore, we have $\AA[h,1]\aa[1]+\cdots+\AA[h,n]\aa[n]=0\aa[1]+\cdots+0\aa[n]=0$ for every $\aa\in \RR^n$. 
But $\rhs[h]\neq 0$, so there is no $\aa\in \RR^n$ satisfying {\bf Problem \ref{pro:1}}.
\qed
\end{proof}

\begin{mydef}
	$\XI,\XJ$ are two sub-sequences of $1,2,\cdots,n$.  $\AIJ$ is called  $(\XI,\XJ)$-projection of $\AA$ (briefly projection of $\AA$) if $\AIJ[i,j]=\AA[\XI[i],\XJ[j]],$ for $i=1,\cdots,|\XI|,j=1,\cdots,|\XJ|$
	\label{def:projection}	
\end{mydef}

\begin{mydef}
	In {\bf Problem \ref{pro:1}}, let  $\HB=\{i\mid \rhs[i]\neq 0 \}$.    $\AIJ$ is called a  computable projection of {\bf Problem \ref{pro:1}} if 
	\begin{enumerate}[(i)]
		\item $\HB \subseteq  \XI$;
		\item $\{i\mid \AA[i,j]\neq 0, j\in \XJ \} \subseteq \XI$.
	 	\item $\{j\mid \AA[i,j]\neq 0, i\in \XI \} \subseteq \XJ$.
	\end{enumerate}   
	\label{def:comput}
\end{mydef}

\begin{mydef}
	$\XI$ is a sub-sequence of $1,2,\cdots,n $. $\rhs_{\XI}$ is called an $\XI$-projection of $\rhs$ if $\rhs_{\XI}[i]=\rhs[\XI[i]]$ for $i=1,\cdots, |\XI|$
	
\end{mydef}

\begin{mydef}
	$\XI$ is a sub-sequence of $1,\cdots, n$ and $\aa$ is a vector such that $|\aa|=|\XI|$. $\lift(\aa, \XI, n)$ denotes a lifting vector where
	\[\lift(\aa, \XI, n)[i]=\begin{cases}
	\aa[j]& \quad \text{ if } \XI[j]=i \\
	0& \quad  \text{ otherwise} 
	\end{cases} \quad \mbox{ for } i=1,\cdots,n\]
\end{mydef}

\begin{lemma}
Let	$\AIJ$ be a computable projection of {\bf Problem \ref{pro:1}}. If there exists a vector $\aa$ satisfying that $\AIJ\aa=\rhs_{\XI}$, then $\lift(\aa, \XI, n) $ is a solution of {\bf Problem \ref{pro:1}}.  \label{lem:lift}
\end{lemma}
\begin{proof}
	Let $\ksi=\lift(\aa, \XI, n)$. In the following we want to prove that $\AA\ksi=\rhs$. W.l.o.g. set $\XI=\{1,2,\cdots,n_1\},\ \XJ=\{1,2,\cdots,m_1\}$.
	
	First, we will prove that $\AA[i]\ksi=\rhs[i]$ for $i=1,\cdots,n_1$. By definition of $\ksi$, $\ksi[i]=0$ for $i>m_1$. Thus, \[\AA[i,1]\ksi[1]+\cdots\AA[i,n]\ksi[n]=\AA[i,1]\ksi[1]+\cdots+\AA[i,m_1]\ksi[m_1]=\rhs[i]\] for $i=1,\cdots,n_1$. 
	
	Second, we will prove that $\AA[i]\ksi=\rhs[i]=0$ for $i=n_1+1,\cdots,n$. Since $\AIJ$ is a computable projection, $\AA[i,j]=0$ for $i> n_1, j\le m_1$. Thus, \[\AA[i,1]\ksi[1]+\cdots\AA[i,n]\ksi[n]=\AA[i,m_1+1]\ksi[m_1+1]+\cdots+\AA[i,n]\ksi[n]\] for $i=n_1+1,\cdots,n$. By the definition of $\lift$, $\ksi[i]=0$ for $i>m_1$. Therefore, 
	\[\AA[i,m_1+1]\ksi[m_1+1]+\cdots+\AA[i,n]\ksi[n]=\rhs[i]=0\] for $i=n_1+1,\cdots,n$.
	
	In summary, $\AA[i]\ksi=\rhs[i]$ for $i=1,\cdots,n$. So  $\lift(\aa, \XI, n) $ is a solution to {\bf Problem \ref{pro:1}}.
	\qed
\end{proof}

\begin{theorem}
	If $\AA_{\XI,\XJ}$ is a computable projection of {\bf Problem 
	\ref{pro:1}}, then the
	system has solution $\ksi$ iff   there exists a vector $\aa$ such that
	$\AIJ\aa=\rhs_{\XI} $.       Furthermore, $\lift(\aa, \XI, n) $ is a solution of {\bf Problem \ref{pro:1}}. 
	\label{the:main}
\end{theorem}

\begin{proof}
	W.l.o.g. we set $\XI=\{1,2,\cdots, n_1\},\ \XJ=\{1,2,\cdots, m_1\} $.
	Let us assume that  $\ksi$ is a solution of {\bf Problem \ref{pro:1}}.   By condition (iii) of {\bf Definition \ref{def:comput}},  we have $A[i,j]=0$ when $i\le n_1$ and $j>m_1$. Thus,
	\[\AA[i,1]\ksi[1]+\cdots+\AA[i,n]\ksi[n]=\AA[i,1]\ksi[1]+\cdots+\AA[i,m_1]\ksi[m_1] =\rhs[i]\]
	for $1,\cdots, n_1$.
	In other words, $\AIJ\ksi_{\XJ}=\rhs_{\XI}$.
	
	 Since   $\AIJ$ is one of computable projections of {\bf Problem \ref{pro:1}}, then applying {\bf Lemma \ref{lem:lift}}, $\lift(\aa,\XJ,n  )$ is one solution of {\bf Problem \ref{pro:1}}.\qed
\end{proof}

\begin{lemma}
	In {\bf Problem \ref{pro:1}},	let $\HB=\{i\mid \rhs[i]\neq 0 \},\ \XI= \{i\mid i\in \reach_{G(\AA)}(\HB),\ i<n\},\  \XJ= \{j-n\mid j\in \reach_{G(\AA)}(\HB),\ j\ge n\}$. Then   $\AIJ$ is a computable projection of {\bf Problem \ref{pro:1}},  if it  has a solution.\label{lem:comp}
	\label{lem:add}
\end{lemma}
\begin{proof}
	Since {\bf Problem \ref{pro:1}} has a solution, we have $\HB\subseteq \XI$ because  otherwise it will conflict with {\bf Lemma \ref{lem:sol}}.  By definition of $\XI$ and $\XJ$, $\XI=\{i\mid \AA[i,j]\neq 0, j\in  \XJ\}$ and $\XJ=\{j\mid \AA[i,j]\neq 0, i\in  \XI\}$.
	In summary,  $\XI,\XJ$ satisfy condition (i) (ii) (iii) of {\bf Definition \ref{def:comput}}.
	 So, 
$\AIJ$ is a computable projection of  {\bf Problem \ref{pro:1}}.\qed
\end{proof}
\begin{corollary}
	In {\bf Problem \ref{pro:1}},	let $\HB=\{i\mid \rhs[i]\neq 0 \},\ \XI= \{i\mid i\in \reach_{G(\AA)}(\HB),\ i < n\},\ \XJ= \{j-n\mid j\in \reach_{G(\AA)}(\HB),\  j\ge n\}$. Then   $\AIJ\yy=\rhs_{\XI}$ has  solution and $\lift(\yy, \XI, n)$ is a solution to	 {\bf Problem \ref{pro:1}},  if 
	{\bf Problem \ref{pro:1}} is feasible. \label{cor:1}   
\end{corollary}
\begin{proof}
When there is an  $\xx$ satisfying {\bf Problem \ref{pro:1}}, 
employing {\bf Lemma \ref{lem:comp}},  $\AIJ$ is a computable projection of $\AA$. By {\bf Theorem \ref{the:main}}, the conclusion holds.\qed	
\end{proof}

\begin{theorem}
	In {\bf Problem \ref{pro:1}},	let $\HB=\{i\mid \rhs[i]\neq 0 \},\ \XI= \{i\mid i\in \reach_{G(\AA)}(\HB), i< n\},\  \XJ= \{j-n\mid j\in \reach_{G(\AA)}(\HB),\  j\ge n\}$. If $\AA$ is a non-singular matrix then there is a unique vector $\yy$ such that  $\AIJ\yy=\rhs_{\XI}$.  \label{the:2}
\end{theorem}
\begin{proof}
Since $\AA$ is a non-singular matrix, {\bf Problem \ref{pro:1}} has a unique solution $\xx$. Employing Corollary \ref{cor:1},  there is a vector $\yy$ such that $\AIJ\yy=\rhs_{\XI}$. Suppose that there is another $\yy'\neq \yy$ such that 
 $\AIJ\yy'=\rhs_{\XI}$. By {\bf Lemma \ref{lem:lift}}, $\xx=\lift(\yy, \XJ, n ), \xx'=\lift(\yy', \XJ, n ) $ are two solutions of {\bf Problem \ref{pro:1}}. In other words, $\AA(\xx-\xx')=0$.   It is easy to check that $\xx\neq \xx' $. Hence, $\AA$ is a singular matrix, which conflicts with the fact that $\AA$ is a non-singular matrix. So, $\AIJ\yy=\rhs_{\XI}$ has  a unique solution. 
\qed	
\end{proof}

\begin{corollary}
	In {\bf Problem \ref{pro:1}},	let $\HB=\{i\mid \rhs[i]\neq 0 \},\ \XI= \{i\mid i\in \reach_{G(\AA)}(\HB),\ i< n\},\  \XJ= \{j-n\mid j\in \reach_{G(\AA)}(\HB),\ j\ge n\}$. If $\AA$ is a non-singular matrix then $|\XI|=|\XJ|$ and $\AIJ$ is a non-singular matrix.
	\label{cor:sym}
\end{corollary}
\begin{proof}
	Let $\BS=\left(\AIJ\right)^T \left(\AIJ\right)$. It is easy to know
	that $\BS$ is a symmetric matrix. Proving  $\BS$ is a non-singular matrix is equivalent to  check that $\xx=\zero$ is a unique solution of   $\xx^T\BS\xx=0$. By {\bf Theorem \ref{the:2}}, $\zero$ is a unique solution of equation system
		$\AIJ\xx=\zero$.  Thus, $\xx=\zero$ is a unique solution of   $\xx^T\BS\xx=0$. In other words,
		$\BS$ is a non-singular matrix and $|\XI|\ge |\XJ|$ and ${\tt rank}(\AIJ)=|\XJ|$.  
		
		Let $b'$ be a vector which satisfies that $\{i\mid b'[i]\neq 0\}=\XJ$. When solving equation $\AA^T\xx=b'$,
		by the same way of above discussion
		we can obtain that $|\XJ|\ge |\XI|$. Therefore, $|\XI|=|\XJ|$ and $\AIJ$ is a  non-singular matrix.
\qed 
\end{proof}

\begin{algorithm}[H]	
	\SetAlgoCaptionSeparator{.}
	\DontPrintSemicolon
	\KwIn{ Matrix $\AA$ and vector $\rhs$   }
	\KwOut{ Vector $\ksi$  which satisfies that $\AA\ksi=\rhs$, or {\bf false} if $\AA\xx=\rhs$ has no solution.}
	\caption{\locS \label{alg:sparse}}
	Let $\HB=\{i\mid \rhs[i]\neq 0 \}$;\;
	Let	$\XI= \{i\mid i\in \reach_{G(\AA)}(\HB), i< n\}, \XJ= \{j-n\mid j\in \reach_{G(\AA)}(\HB), j\ge n\}$;\;
	\label{reach}
	\If{$\HB\not\subseteq \XI$ }{
		\Return {\bf false};\tcc{By {\bf Lemma \ref{lem:sol}}, if $\HB\not\subseteq \XI$ then problem has no solution. }}
	\tcc{By Theorem \ref{the:main}, $\AIJ$ can be replaced by any computable projection. }
	\tcc{By Corollary \ref{cor:sym}, $\AIJ$ is a non-singular square matrix.}
	\tcc{We can use \KLU\ solver to solve the following linear equation $\AIJ\yy=\rhs_{\XI}$.}
	Let $\aa$ be a vector satisfying $\AIJ\aa=\rhs_{\XI}$;\; 
	\tcc{By Corollary \ref{cor:1}, $\lift(\aa, \XI, n )$ is a solution of $\AA\xx=\rhs$ }
	\Return $\lift(\aa, \XI, n )$;\;
\end{algorithm}

\subsubsection{Impoving algorithm \locS\ during iteration}
Computing reachable edges for a  given node set $\HB$  is a key step  of \locS. Although computing reachable set of a given graph is of linear complexity,
 but in this case we need to construct a new graph per iteration.  Luckily, by  {\bf Theorem \ref{the:main}}, 
we can use any computable projection 
to replace $\AIJ$.  Thus we can present a fast method instead of explicitly computing $\reach_{G(\AA)}(B)$.
  As discussed in Section  \ref{sec:incremental}, $G(\UU_{k+1})$ is very similar to  $G(\UU_k)$, so this approach is feasible.
Thus, we utilize the information of  $G(\UU_k)$ to construct computable projection of $G(\UU_{k+1})$.


\begin{note}
	For a graph $G$,
  $\vv(G)$ denotes set of   nodes in  $G$.
\end{note}

\begin{mydef}
$G$ is a graph with  nodes $1,\cdots, 2n$.  $\{G_1,\cdots,G_s\}$ are  graphs such that $\vv(G_i)\subseteq \{1,\cdots, 2n\}$. We call 
$\{G_1,\cdots,G_s\}$ an \emph{over disjoint cover} of $G$ if \label{def:over}
\begin{enumerate}[(i)]
	\item $\vv(G_i)\cap\vv(G_j)=\emptyset$ for $i\neq j$;
	\item there is $G_i$ such that $e\in G_i$ for edge $e\in G$.
\end{enumerate}\label{def:cover}
\end{mydef}
For a given graph $G$, $G$'s 	 different connected components  $\HC_{1},\cdots,\HC_{s}$ is one of its 
{\em over disjoint cover}. 
\begin{theorem}
In {\bf Problem \ref{pro:1}},  let $\HB=\{i\mid \rhs[i]\neq 0 \}$, $\{G_1,\cdots,G_s\}$ be an {\em over disjoint cover} of $G(\AA)$.
Let $E=\{(i,j)\mid (i,j)\in G_i, \vv(G_i)\cap \HB\neq \emptyset  \}$. Let $\XI=\{i\mid (i,j)\in E\}, \XJ=\{j-n\mid (i,j)\in E\}$. If {\bf Problem \ref{pro:1}} has a solution then $\AA_{\XI, \XJ}$ is a computable projection.
\label{the:cover}
\end{theorem}
\begin{proof}
Since	{\bf Problem \ref{pro:1}} has a solution,
by {\bf Lemma \ref{lem:sol}}, there is $(i,j)\in G(\AA)$ for $i\in \HB$. Thus, $\HB\subseteq \XI$.

Assume that $\{i\mid A[i,j]\neq 0, j\in \XJ \} \not \subseteq \XI$, in other words, there is a $h$ such that $h\in \left (\{i\mid A[i,j]\neq 0, j\in \XJ \} \setminus \XI\right)$. In other words, $h\in \{i\mid A[i,j]\neq 0, j\in \XJ \} $ and $ h\not \in \XI$.
Let  $t\in\XJ$ such that $\AA[h,t]\neq 0$ and $(h,t+n)\in G_v$. By definition of  $\XJ$ there is an edge $(u,t+n)\in E$ because $t\in \XJ$. 
 
 By {\bf Definition \ref{def:cover}}, all $G(\AA)$'s edges whose nodes contain  $t$ must be completely contained in $G_v$. Therefore, the assumption can not hold. Hence, $(u,t+n) \in G_v$.
 
 We want to prove that $(u,t+n)\not \in G_v$.　We  prove it by contradiction. If $(u,t+n) \in G_v$, then by  definition of $E$ all the edges of $G_v$ will belong to $E$, in particular,  $(h,t+n)\in E$. Thus, $h\in \XI$. This 
conflicts with $h\not \in \XI $, so  $(u,t+n)\not \in G_v$. 

 Thus,
$\{i\mid A[i,j]\neq 0, j\in \XJ \}  \subseteq \XI$, by the same way we can prove (iii) of {\bf Definition \ref{def:comput}}. Hence $\AA_{\XI, \XJ}$ is a computable projection.
 \qed
\end{proof}
Theorem \ref{the:cover} provides a new approach to constructing computable projection.  This approach can be used to replace the computation of $\reach_{G(\AA)}(\HB)$ in line \ref{reach} in \locS.


\begin{lemma}
	$\UU,\UU'$ are two $n\times n$  matrices. $\{G_1,\cdots,G_s\}$  is an over disjoint cover of $G(\UU)$.
	Denote the set of nonzero elements in matrix  $(\UU-\UU')$ by    $\{(\UU-\UU')[i_k,j_k]\mid k=1,\cdots, m\}$.
	We iteratively update   $\{G_1,\cdots,G_s\}$  by the following operation:
	merging $G_i,G_j$ to $G'$ where $ G'=G_i\cup G_i\cup \{(i_k,j_k+n)\}$  if there is a link $(i_k,j_k+n)$  connecting  $G_i$ and $G_j$.
	Let $G'_1,\cdots,G'_{s'}$ be    finally resulted graphs  of the above iteration.    
	Then $\{G'_1,\cdots,G'_{s'} \}$ is an over disjoint cover of $G(\UU')$.  \label{lem:cover}
\end{lemma}
\begin{proof}
	We prove it by induction. 
	
	When $m=1$. If both $\UU[i_1,j_1]$ and $\UU'[i_1,j_1]$ are nonzeros, then $\{G'_i \}=\{G_i \}$ and $G(\UU)=G(\UU')$.
	Thus, conclusion holds.

	If only $\UU[i_1,j_1]\neq 0$, then  $\{G'_i \}=\{G_i \}$ and $   \vv(G(\UU'))\subseteq \vv(G(\UU))$.  Thus conclusion holds.
	
	If only $\UU'[i_1,j_1]\neq 0$, then $G(\UU')$ only has one more link $(i_1,j_1+n)$  compared  with $G(\UU)$. If there exist $G_i,G_j$ such that $i_1\in \vv(G_i), j_1\in \vv(G_j)$, 
	 then   merge $G_i,G_j$ into a graph $G'=G_i\cup G_i\cup \{(i_1,j_1+n)\}$.  It is easy to check that $\left( \{G_1,\cdots,G_s\}\setminus \{ G_i,G_j\} \right)\cup\{G'\}$ satisfies (i)-(ii) of {\bf Definition \ref{def:over}}.  Hence the conclusion  holds when $m=1$.

	Assuming that the conclusion holds when $m\le K_1$. When $m=K_1+1$, let $\UU''$ be a matrix such that 
	$(\UU''-\UU')$ has only one nonzero element  $(\UU''-\UU')[i_1,j_1]$,  and $\UU[i_1,j_1]=\UU''[i_1,j_1]$. It is easy to know that matrix $(\UU''-\UU)$ has only  $K_1$   nonzero elements. 
	By assumption, $\UU'$'s \emph{over disjoint cover} can be constructed from $\UU''$'s, which can be constructed from  $\UU$'s .
	
	In summary, conclusion holds for any $m\ge 0$.\qed
\end{proof}	

Through {\bf Lemma \ref{lem:cover}}, we can construct \emph{over disjoint cover of} $G(\UU_{k+1})$ from $G(\UU_k)$'s. This can be used to
fast compute  computable projection in $(k+1)$th iteration from $k$th's.

\subsection{Incremental Change Property of $\UU_k$'s  Nonzero Pattern}
\label{sec:incremental}

Fast solving  equations (\ref{eq:simple}), (\ref{eq:simplexcoef}) and (\ref{dual:simple}) is a feasible way of improving efficiency of \SMCG. 
In the following section we will first  give 
an Algorithm \incS\  which utilizes the sparse and incremental change properties of matrices and vectors occurring in two consecutive equation systems to fast solve target equation. Second we will describe the three interesting phenomenons during \SMCG's iteration. And these phenomenons can let us directly employ \incS\ instead of other solvers. By this way, we can fast solve equations
(\ref{eq:simple}), (\ref{eq:simplexcoef}) and (\ref{dual:simple}).
\subsubsection{Fast method of solving similar linear equations}
\label{sec:incs}
\begin{problem}
	$\AA$ is a non-singular sparse matrix.
	$\AA,\AA'$ are two very similar  matrices, in other words, $(\AA-\AA')$ has few nonzero elements.
	$\rhs,\rhs'$ are  two very similar vectors. When there is a vector $\ksi'$ such that
	\begin{equation}
	\AA'\ksi'=\rhs',
	\label{eq:first}
	\end{equation}
	we want to give an  efficient algorithm to solve equation
	\begin{equation}	
		\AA\xx=\rhs.
	\label{eq:second}
	\end{equation}
	\label{pro:2}
\end{problem}

Assume that $\ksi $ is a solution of (\ref{eq:second}).
Because the coefficient matrices and right-hand-sides  of (\ref{eq:first}) and (\ref{eq:second}) are very similar. It is reasonable to 
believe that  $\ksi,\ksi'$  are very similar. Thus, we 
only need to compute the different part  for these two solutions when solving (\ref{eq:second}). The concrete algorithm of this idea is listed in Algorithm \incS.

The outline of Algorithm \incS\ is as follows: When we want to solve equation (\ref{eq:second})  when there is a vector $\ksi'$ satisfying (\ref{eq:first}).  In this case, we can believe
that   $(\rhs-\AA\ksi')$ is a sparse vector.  So, we  firstly use Algorithm \locS\ to find a vector $\Delta\ksi$  which is a  solution of equation $\AA\xx=(\rhs-\AA\ksi')$. It is easy to check that $\AA(\ksi'+\Delta\ksi)=\rhs$.

\begin{algorithm}
	\SetAlgoCaptionSeparator{.}
	\DontPrintSemicolon
	\KwIn{ Two matrices $\AA,\AA'$ and three vectors $\rhs,\rhs',\ksi'$ where $\AA$ is a non-singular sparse matrix, $\AA-\AA'$ is a sparse matrix, $\rhs-\rhs'$ is a sparse vector and $\AA'\ksi'=\rhs'$    }
	\KwOut{One vector $\ksi$  which satisfies that $\AA\ksi=\rhs$ }
	\caption{\incS\label{alg:inc}}
	\If{$\ksi'=0$}{
		\Return \  \locS($\AA$,$\rhs$);\;
	}
	Let $\Delta\rhs=\rhs-\AA\ksi'$;\;
	\tcc{Employing Algorithm \locS\ to solve this equation $\AA\xx=\Delta\rhs$ }
	Solving $\AA\xx=\Delta\rhs$ to obtain a solution $\Delta\ksi$;\;
	Let $\ksi=\ksi'+\Delta\ksi$;\;
	\Return $\ksi$;\;
\end{algorithm}

\subsection{Incremental Change Property of $\UU_k$'s  Nonzero Pattern}
\label{sec:inc}
In this section we will list the three interesting phenomenons.
Firstly, matrices $\UU_{k}, \UU_{k+1}$ have  little difference. Secondly,  right-hand-sides  in
(\ref{eq:simple}) and (\ref{dual:simple}) also have  little   difference between $k$th and $(k+1)$th iteration.  Thirdly,	 right-hand-side of (\ref{eq:simplexcoef}) is very sparse.  All of these phenomenons indicate that \locS\ and \incS\ are the proper solvers for (\ref{eq:simple}), (\ref{eq:simplexcoef}) and (\ref{dual:simple}).
So, we can use Algorithm \incS\ to  quickly construct solutions of  (\ref{eq:simple}), (\ref{eq:simplexcoef}) and (\ref{dual:simple}).

As described in {\bf Model \ref{model:structed}}, $\UU_k$ is entirely defined by $\KK,\SS_k$ and their elements' order. Therefore, changing order of $\KK,\SS_k$'s elements can give a better  incremental property of matrices  and  vectors occurring in iteration. In the following description,  without special statement the same elements of $\KK$ have the same matrix index  between $k$th and $(k+1)$th iteration, and the same  elements of $\SS_k$ and $\SS_{k+1}$ have the same matrix index. 
For obtaining   incremental property, we firstly redefine transition system rules as follows:
\begin{framed}
	
	\begin{enumerate}		
		\item When the entering variable is a link $e^*$  :
		\begin{enumerate}
			\item The same as Fig. \ref{sys:trans} 1-(a)
			\item The same as Fig. \ref{sys:trans} 1-(b)
			\item When the leaving variable is a link $e$:
			
			Let $ \NN_{k+1}=\{\NN_k\cup\{e^*\}\}\setminus\{e\}$. {\bf And let index of  $e$ in $\SS_{k+1}$ be the 	same as $e^*$ in $\SS_k$, and the other  links' indices are kept the same between  $\SS_k$ and $\SS_{k+1}$.}
			
		\end{enumerate} 
		\item When the entering variable is a $p'_j$:
		\begin{enumerate}
			\item The same as Fig. \ref{sys:trans} 2-(a) 
			\item The same as Fig. \ref{sys:trans} 2-(b) 
			\item When the leaving variable is a link $e$:
			
			Let $\QQ_{k+1,j}=\QQ_{k,j}\cup \{ p_j'\}, \NN_{k+1}=\NN_k\setminus\{ e\}$.
			{\bf $p'_j$ is a path corresponding  to saturate link $e$. Append $e$ to $\SS_k$
				to obtain $\SS_{k+1}$. And other links' indices are kept the same between $\SS_k$ and $\SS_{k+1}$.}
			
		\end{enumerate}
	\end{enumerate}
	\captionof{figure}{Transition system rules for iteration in  Algorithm \SMCG. \label{sys:stabletrans}}
\end{framed}

\paragraph{First phenomenon.} From transition system rules in Fig. \ref{sys:stabletrans}, we can find that size of matrix $\UU_{k+1}$ only has three cases, i.e.
$|\UU_k|, |\UU_k|+1 and |\UU_k|-1 $. Below we will discuss relation between $\UU_{k+1}$ and $\UU_k$ under these three cases.

First, when $|\UU_{k+1}|=|\UU_k|$,
 we will give a useful fact that the number of $\left(\UU_{k+1}-\UU_k\right)$'s columns which has nonzero elements is very few. 

\begin{lemma}
 In Note \ref{note:g}, if matrix $\CC'$ only has one column corresponding to commodity $i$  different from  $\CC_k$, then there are at most $|\QQ_{k,i}|$ columns of $\CC'\BB_k-\DD_k$ different from $\UU_k$. \label{lem:less}
\end{lemma}
\begin{proof}
	By the definition of $\BB_k$ in (\ref{eq:factor}), every column of matrix $\BB_k$ corresponding to a commodity $i$. Let  $c$ be $i$th column of $\CC_k$. By the order of $\KK$, $c$ is corresponding to primary path $p_{k,i}$.   Let $\beta$ be $j$th  column   of  matrix $\BB_k$    which is corresponding to a commodity $i$. By equation (\ref{eq:factor}), the form of $\beta$ is as follows
	\[\beta[j]=
	\begin{cases}
	1 \mbox{ if } j=i,\\
	0 \mbox{ otherwise.} 
	\end{cases}
	\]
	 Hence, in product $\CC_k\BB_k$ only column $c$ of matrix $\CC_k$ affects   $\beta$.  And $\CC_k \beta=c$ is the $j$th column of product $\CC_k\BB_k$. Thus, in product $\CC_k\BB_k$,  $i$th column of $\CC_k$ only affect columns  corresponding to $\QQ_{k,i}$.  Therefore, changing column $c$'s value only changes  $|\QQ_{k,i}|$ columns' values of product $\CC_k\BB_k$, since the number of
	  $\BB_k$'s columns  are the same as $\beta$ is
	  $|\QQ_{k,i}|$. \qed
\end{proof}
Employing {\bf Lemma \ref{lem:less}}, we can give relation between  $\UU_{k+1}$ and $\UU_{k}$ as follows:
\begin{enumerate}
	\item   If we use rule 1-(c) of Fig. \ref{sys:stabletrans} during iteration and let $e$ have the same index of $\SS_{k+1}$ as $e^*$ in $\SS_k$, then 
	$\UU_{k+1}$ has at most  \[ |\bigcup_{\left(i=1,\pi^{p_{k,i}}_e+\pi^{p_{k,i}}_{e^*}= 1\right) }^{\LL}\QQ_{k,i}| +|\{p\mid (\pi^p_e+\pi^p_{e^*})= 1,  p\in \bigcup_{i=1}^{\LL}\QQ_{k,i}  \}|\]
	columns different from $\UU_k$.
	
	\item  If  we use  rule 2-(a) of Fig. \ref{sys:stabletrans} during iteration, then 
	$\UU_{k+1}$ at most has $|\QQ_{k,i}|$ columns different   from  $\UU_k$. 

	\item If we use rule 2-(b)  of Fig. \ref{sys:stabletrans}
	during iteration, then $\UU_{k+1}$ at most has $1$ column different from $\UU_k$.
\end{enumerate}
In summary, when size of $\UU_{k+1}$ equals  size of $\UU_{k}$, matrix $(\UU_{k+1}-\UU_k)$ has   few nonzero columns.

Second, when $|\UU_{k+1}|=|\UU_k|+1$.
Only after transition system rule 2-(c), this case can occur. And  matrix $\UU_{k+1}$  can be written as
 \[\UU_{k+1}=\begin{bmatrix}
 \UU_{k}& \theta \\[0.3cm]
 \rho^T  & a\\
 \end{bmatrix}\]
 where $\theta, \rho$ are vectors and $a$ is scalar. So $\UU_{k+1}$ and $\UU_k$ are very similar.
 
Third, when $|\UU_{k+1}|=|\UU_k|-1$. Only after transition system rule 1-(a), 1-(b), this case can occur. And  relation between  $\UU_{k+1}$ and $\UU_k$ can be written as
  \[\UU_{k}=\begin{bmatrix}
  \UU^{(1)}_{k}& \theta_1 & \UU^{(2)}_{k} \\[0.3cm]
  \rho_1^T  & a   & \rho_2^T  \\[0.3cm]
  \UU_k^{(3)} & \theta_2 & \UU_k^{(4)} \\
  \end{bmatrix},\quad \UU_{k+1}= \begin{bmatrix}
  \UU^{(1)}_{k}& \UU^{(2)}_{k} \\[0.3cm]
  \UU_k^{(3)} & \UU_k^{(4)} \\
  \end{bmatrix} \] 
 where  $\UU^{(i)}_{k}$ are matrices, $\theta_i, \rho_i$ are vectors and $a$ is a scalar. So $\UU_{k+1}$ and $\UU_k$ are very similar.

 \paragraph{Second phenomenon.}  By the same analysis procedure as above, we can easily check that
 \begin{enumerate}
 	\item $\AA_k$ and $\AA_{k+1}$
 	are similar;
 	\item $\cc_k$ and $\cc_{k+1}$ are similar;
 	\item $\bb_k$ and $\bb_{k+1}$ are also similar.
 \end{enumerate}
 Thus,  right-hand-sides  in (\ref{eq:simple}) and (\ref{dual:simple}) in $k$th and $(k+1)$th iterations are similar too.

  \paragraph{Third phenomenon.} 	In equation (\ref{eq:simplexcoef}), if associate entering variable of  $\bb$ is an edge $e^*$, 
   the form of $\beta$ is as follows
   \[\beta[i]=
   \begin{cases}
   1 \mbox{ if } i \mbox { is the index associated with } e,\\
   0 \mbox{ otherwise.} 
   \end{cases}
   \] By the discussion in  the beginning of Section \ref{section:sparse}, 
  every column of $\CC_k$ is a sparse column in general. So, $\CC_k\bb_{\KK}-\bb_{\SS_k}$ is also  a sparse vector in general at {\bf Model \ref{model:structed}}.
  
   Otherwise,
   $\bb$ is a basic vector associated with a path $p$. Then, in {\bf Model \ref{model:structed}} $\bb_{\KK}$ only has one nonzero element which is $1$.
   And $\bb_{\SS_k}$ is also sparse in general by discussion in the  beginning of Section \ref{section:sparse}. Thus, $\CC_k\bb_{\KK}-\bb_{\SS_k}$
   is the subtraction of two sparse vectors, so is also  a sparse vector in general.

 \subsection{Fast Solving Equations During Iteration}
 \label{sec:qcg}
 By   discussion in Section \ref{sec:inc}, coefficient matrix and right-hand-side of  (\ref{eq:simplexcoef}) are both sparse.
 Hence, we can employ algorithm \locS\ to solve  (\ref{eq:simplexcoef}).
 
In addition, coefficient matrices and right-hand-sides of  (\ref{eq:simple})  and (\ref{dual:simple}) are very similar between $k$th and $(k+1)$th iteration. So, in the following we will employ algorithm \incS\
 to solve equations in $(k+1)$th iteration from solution of $k$th ones.

\begin{enumerate}[(1)]
	\item When $|\UU_{k+1}|=|\UU_{k}|$, we directly employ \incS\ to solve (\ref{eq:simple}) and (\ref{dual:simple}).
	\item When  $|\UU_{k+1}|=|\UU_{k}|+1$, we firstly extend solution of corresponding equation of $k$th by setting the element of new index as $0$. After this extension, we employ \incS\ to solve (\ref{eq:simple}) and (\ref{dual:simple}).
	\item When $|\UU_{k+1}|=|\UU_{k}|-1$,  we firstly narrow solution of  corresponding equation of $k$th by  deleting element of lacking index.
	After this narrowing,  we employ \incS\ to solve (\ref{eq:simple}) and (\ref{dual:simple}).
\end{enumerate}


\section{Experiments}
\label{sec:exp}

\subsection{Environment}
The algorithm of \SMCG\ is implemented as a \texttt{C++} program. Compilation was done using {\tt g++} version 5.4.0 with optimization flags -O2.
We use latest \LAP\ (version 3) and latest \KLU\  which is contained in tool {\tt SuiteSparse}  4.5.6\footnote{ \url{http://faculty.cse.tamu.edu/davis/suitesparse.html}}.
All tests  are done  on a 64-bit Intel(R) Core(TM) i5 CPU 7400 @ 3.00GHz with 8GB RAM memory and Ubuntu 16.04 GNU/Linux.

We use  \incCG,  \kluCG\ and \lapCG\ to denote implementations of \SMCG\ with  \incS, \KLU\ and \LAP\  as linear equation solver, respectively. In other words, except for linear equation solver,  the other parts of \incCG,\ \kluCG\ and \lapCG\   are the  same.

Random test cases are created by generator $R(n)$ where $n$ is the number of nodes. The average node degree (sum of in degree and out degree)  is $10$.
  Each edge is generated by two random integers between $1$ and $n$ as its source and target node indices.
The edge capacity is a random integer between $1$ and $300$ and edge weight  is a random integer between 
$1$ and $10$. The source and target  indices of commodity are two random integers between $1$ and $n$. Commodity demand
is a random integer between $1$ and $100$. Every case has $1000$ commodities.
 
\begin{table}[!h]
	\tiny
	\centering
	\label{tab:other}
	\begin{tabular}{| c| c  c  c| c c  c| c c c |}
		\hline
		\multicolumn{1}{|c|}{Case}& \multicolumn{3}{c|}{Total time(s)} &\multicolumn{3}{c|}{Shortest path computing time(s)}& \multicolumn{3}{c|}{Linear equation solving time(s)}\\
		\hline
		&\incCG&\kluCG &\lapCG  &\incCG&\kluCG & \lapCG &\incCG&\kluCG &\lapCG  \\
		\hline  
	    $R(1000)$&	1538.51&1831.76&	29438.10&	178.65&	177.50&	188.56&	721.23&	1431.20&	27392.50 \\ 	\hline
	    $R(1500)$&	625.76&	801.76&	18970.80&	171.21&	176.44&	184.13&	225.17&	503.32&	17690.60\\ 	\hline
	    $R(2000)$&	258.90&	301.54&	5720.87&	146.49&	154.18&	145.39&	50.48&	105.68&	5219.57\\ 	\hline
	    $R(3000)$&	180.84&	187.72&	2157.97&	150.84&	153.10&	153.97&	12.17&	22.26&	1888.38 \\ 	\hline
	    $R(5000)$&	158.91&	161.67&	609.58&	151.81&	152.48&	152.80&	2.14&	4.56&	425.30 \\ 	\hline
	    $R(7000)$&	200.37&	201.93&	631.31&	194.18&	194.16&	198.87&	1.61&	3.34&	404.67 \\ 	\hline
	    $R(9000)$&	219.11&	217.10&	606.23&	213.44&	210.16&	231.69&	1.31&	2.73&	350.03 \\ 	\hline
	    $R(11000)$&	246.31&	251.99&	622.97&	240.89&	245.40&	265.44&	1.12&	2.30&	334.39 \\ 	\hline
	    $R(13000)$&	266.82&	278.51&	618.77&	261.25&	271.96&	287.63&	1.02&	2.03&	309.15 \\ 	\hline
	    $R(15000)$&	270.44&	271.01&	491.10&	265.30&	265.19&	277.14&	0.74&	1.46&	197.66 \\ 	\hline
	    $R(17000)$&	315.24&	314.90&	590.03&	309.22&	308.13&	324.42&	0.84&	1.65&	246.35 \\ 	\hline
	    $R(19000)$&	370.37&	375.79&	706.79&	363.69&	368.36&	372.57&	0.91&	1.73&	312.02 \\ 	\hline
	    $R(21000)$&	366.81&	374.80&	604.03&	360.39&	367.75&	350.00&	0.76&	1.46&	235.44 \\ 	\hline
	    $R(23000)$&	393.72&	395.79&	600.08&	387.01&	388.48&	357.91&	0.72&	1.37&	223.79 \\ 	\hline
	    $R(25000)$&	383.42&	386.20&	544.56&	377.12&	379.47&	386.96&	0.55&	1.05&	142.73 \\ 	\hline
	    $R(27000)$&	459.25&	460.63&	690.73&	451.89&	452.79&	460.94&	0.68&	1.24&	211.07 \\ 	\hline
	   $R(29000)$&	484.89&	484.76&	701.85&	477.21&	476.70&	488.44&	0.64&	1.16&	195.01 \\ 	\hline
	   $R(31000)$&	537.48&	539.50&	738.52&	529.59&	531.18&	524.11&	0.63&	1.14&	195.86 \\ 	\hline
	   $R(33000)$&	560.34&	588.66&	758.42&	552.50&	580.30&	563.80&	0.58&	1.05&	177.01 \\ 	\hline
	   $R(35000)$&	618.65&	623.70&	817.98&	610.31&	614.62&	605.78&	0.61&	1.14&	193.48 \\ 	\hline
	   $R(37000)$&	692.31&	701.09&	886.41&	683.28&	691.64&	644.37&	0.64&	1.11&	221.87 \\ 	\hline
	   $R(39000)$&	677.80&	699.56&	812.32&	668.49&	689.65&	609.01&	0.60&	1.13&	184.44 \\ 	\hline
	\end{tabular}
		\captionof{table}{Different parts time comparison   of \incCG,\ \kluCG\ and \lapCG.\label{tab:compare}}
\end{table}
In Table \ref{tab:compare}, you can see that  shortest path computing  and linear equation solving  are two major  time consuming parts 
of implementations. And except for  $R(1000)$, $R(1500)$ and  $R(2000)$, 
 the total time of \incCG,\ \kluCG\ and \lapCG\  are almost equal to 
sum of shortest path computing time and linear equation solving time. And the shortest path computing time of different implementations 
are almost the same. In addition, considering total time,  when linear equation solving is dominating  part, \incS\ will achieve high speedup. For example,  we can see in case $R(1000)$ using \incS\ instead of \LAP\ will achieve $19\times$ improvement. On the other hand, when   linear equation solving costs less time, \incS\ can reduce 
 linear equation solving  to a negligible fraction. For example, in case $R(7000),\cdots,R(39000) $, using \incS\ instead of \LAP\
 will reduce  linear equation solving time to less than $1\%$ of the total.

\begin{figure}[!h]
			\centering
			\includegraphics[width=0.8\textwidth]{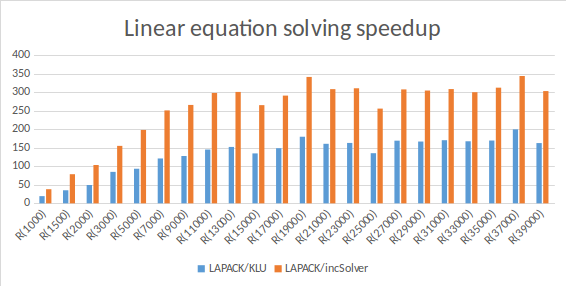}
			\caption[Network2]%
			{{\small  \KLU\  and \incS\  speedup compared with \LAP\ under different random cases}}    
			\label{fig:speedup}
\end{figure}
In Fig. \ref{fig:speedup}, we can see that \incS\ outperforms \KLU\ and \LAP\ on all the test cases. Among these 
cases, \KLU's speedup is between $19$ and $199$ compared with \LAP, while \incS\ achieves  a  speedup from   $37$ to $341$. When 
comparing \incS\ with \KLU, \incS's speedup is between $1.7$ and $2.1$.  As \incS\ is a prototype implementation, we believe that 
 \incS\ has great potential for improvement.  
  \section{Conclusion}
  \label{section:conc}
  In  this paper, for speeding up linear equation solving part in column generation for multi-commodity flow problem, firstly,  we use \emph{transition system} view to describe the procedure of \emph{column generation}.  This view
  can help us better understand the procedure of column generation and it also helps us 	
  conveniently present following improvement.  Secondly, we discuss the sparse property 
  of  coefficient matrix. In the \SMCG\ the average  number of nonzero coefficient in each  row of  coefficient matrix is very few. In our test it is less than $5$, even when the dimension of matrix is more than $1000$.  Finally,  we present  two algorithms. The first is a fast algorithm \locS\  (for \emph{localized system solver}) which can reduce the number of variables in solving a linear equation system  when both the coefficient matrix and right-hand-side   are sparse. The other is an  algorithm \incS\  (for \emph{incremental system solver})  which utilizes similarity  during the iteration in solving a linear equation system. All algorithms can be used in  column generation
  of multi-commodity problem. Preliminary numerical experiments show that the algorithms are significantly  faster than existing algorithms.  For example, under random test cases \incS\ delivers up to  $ 341\times$ (from $37\times$) improvement in the linear equation solving part compared with \LAP.  In addition, considering total time,  when linear equation solving is dominating  part, \incS\ will achieve high speedup. For example in some tests   using \incS\ instead of \LAP\ will achieve $19\times$ improvement. On the other hand, when   linear equation solving costs less time, \incS\ can reduce 
  linear equation solving  to a negligible fraction.  For example  in some cases using \incS\ instead of \LAP\
  will reduce  linear equation solving time to less than $1\%$ of the total.
 \section{Acknowledgements}
 The authors would like to thank Prof. Zongyan Qiu who helps to improve the presentation of the article.

\bibliographystyle{spmpsci}
\bibliography{cgmcf}

\end{document}